\documentclass[12pt]{article}
\usepackage[T1]{fontenc}
\usepackage{amsfonts}
\usepackage{amssymb}
\usepackage{amsmath}
\usepackage{graphicx}
\usepackage{amsbsy}
\usepackage{theorem}
\usepackage[normalem]{ulem}
\usepackage{color}
\def\Section{\setcounter{equation}{0}\section}
\newtheorem{theorem}{Theorem}[section]
\newtheorem{lemma}[theorem]{Lemma}

\newtheorem{definition}[theorem]{Definition}
\newtheorem{proposition}[theorem]{Proposition}
\newtheorem{remark}[theorem]{Remark}

\newtheorem{assumption}[theorem]{Assumption}
\def\thetheorem{\thesection.\arabic{theorem}}
\def\thesection{\arabic{section}}

\def\theequation {\thesection.\arabic{equation}}
\def\beq{\begin{equation}\displaystyle}
\def\eeq{\end{equation}}
\def\bel{\begin{equation} \displaystyle \begin{array}{l} }
\def\eel{\end{array} \end{equation} }
\def\bell{\begin{equation} \displaystyle \begin{array}{ll}  }
\def\eell{\end{array} \end{equation} }

\def\bea{\begin{eqnarray}}
\def\eea{\end{eqnarray} }
\def\bean{\begin{eqnarray*}}
\def\eean{\end{eqnarray*} }
\newenvironment{proof}{\noindent{\bf Proof.~}}
{{\mbox{}\hfill {\small \fbox{}}\\}}
\catcode`@=11
\renewcommand\appendix{\bigskip {\noindent \Large \bf Appendix}
  \setcounter{section}{0}%
  \setcounter{subsection}{0}%
\setcounter{equation}{0}%
\setcounter{theorem}{0}%
\def\thetheorem{A.\arabic{theorem}}
\def\theequation {A.\arabic{equation}}}
\catcode`@=12
\catcode`@=11
\def\thebibliography#1{\bigskip \noindent {\Large \bf References}\sloppy\nolinebreak
 \list
 {[\arabic{enumi}]}{\settowidth\labelwidth{[#1]}\leftmargin\labelwidth
 \advance\leftmargin\labelsep
 \usecounter{enumi}}
 \def\newblock{\hskip .11em plus .33em minus -.07em}
 \sloppy
 \sfcode`\.=1000\relax}
\catcode`@=12
\newcounter{clegende}
\def\div{{\rm div \;}}

\def\NN{\mathbb{N}}

\def\RR{\mathbb{R}}
\def\LL{\mathbb{L}}

\def\DD{\mathbb{D}}

\def\ds{\displaystyle}

\def\bs{\bigskip}

\def\eps{\varepsilon}
\def\bar#1{{\overline #1}}
\def\fref#1{{\rm (\ref{#1})}}

\def\R2+{\RR ^2_+}

\catcode`@=11
\def\supess{\mathop{\operator@font Sup\,ess}}
\catcode`@=12

\def\bar#1{{\overline #1}}
\def\fref#1{{\rm (\ref{#1})}}

\def\eps{\varepsilon}

\def\calN{{\cal N}}

\def\calE{{\cal E}}

\def\calF{{\cal F}}

\def\calD{{\cal D}}

\def\Epsilon{\mbox{\Large $\epsilon$}}

\def\pa{\partial}

\def\bet{s}

\begin{document}

\bs
\begin{center}
{\Large \bf Numerical simulations of an energy-transport model
for partially quantized particles.}

\bs

{\large P. Pietra$^a$ and N. Vauchelet$^{b,c,}$\footnote{{\em Corresponding author.}}
} 

\bs

{\footnotesize $^a$ Istituto di Matematica Applicata e Tecnologie 
Informatiche ``Enrico Magenes'', CNR,\\ Via Ferrata 1, 27100 Pavia, Italy}

{\footnotesize $^b$ UPMC Univ Paris 06, UMR 7598, Laboratoire Jacques-Louis Lions, F-75005, Paris, France \\
$^c$  CNRS, UMR 7598, Laboratoire Jacques-Louis Lions, F-75005, Paris, France}

\smallskip
{\footnotesize {\em E-mail addresses:} {\tt paola.pietra@imati.cnr.it}
  ; {\tt vauchelet@ann.jussieu.fr}. }  \\

\end{center}

\medskip

\begin{abstract}
A coupled quantum-classical model describing the transport of electrons
confined in nanoscale semiconductor devices is considered.
Using the subband decomposition approach allows to separate the transport
directions from the confinement direction.
The motion of the gas in the transport direction is assumed to be classical.
Then a hierarchy of adiabatic quantum-classical model is obtained,
leading to subband SHE and
energy-transport models, with explicit expression of the diffusion
coefficients. The energy-transport-Schr\"odinger-Poisson model is
then used for the numerical simulation of the transport of the
electron gas in an ultra-scaled Double-Gate-MOSFET.
\end{abstract}

\bs



\bs

\Section{Introduction}

In nanoscale semiconductor devices, electrons might be extremely confined
in one or several directions refered to as the confining direction.
This leads to a partial quantization of the energy
which can be modelled by the subband decomposition method \cite{polizzi}.
This subband decomposition approach allows to separate the confinement
direction from the transport direction.
Thanks to the reduction of the dimension of the
transport problem, the computational gain is significant.
In the confined direction electrons behave like waves;
the system is at thermodynamical equilibrium and it is described by a
statistical mixture of eigenstates of a Schr\"odinger-Poisson system.
In the transport direction the transport can be of classical
\cite{ddsp,nbafmnegulescu} or quantum nature \cite{polizzi}.

Here, we are interested in deriving adiabatic
quantum--classical models accounting for thermal effects, aiming at
accurate and efficient numerical simulation of confined devices. 
In \cite{bourgade} several spherical harmonic expension
(SHE) models incorporating quantum effects are proposed. 
However, with their strategy the obtained models have a complicated 
non-local structure which is not suitable for numerical purposes.
Quantum energy--transport and quantum drift-diffusion models have been 
derived in \cite{QET} using the strategy of quantum moments, as well
as in \cite{jungel2}. These models 
involve a quantum chemical potential that depends on the density 
in a non--local way.
In this work we follow the complementary strategy proposed in
\cite{nbafmnegulescu} where the subband model is derived first, then
a diffusive approximation of the adiabatic Boltzmann equation is performed
to obain coupled quantum--fluid models (spherical harmonic expansion and 
energy--transport).
The subband energy--transport (ET) model in \cite{nbafmnegulescu}, directly 
derived from the Boltzmann equations (as in e.g. \cite{nbaldsg} for the 
classical case), is, however, not immediately suited 
for numerical simulations, since the diffusion coefficients are not given
in explicit form and, moreover, the energy relaxation term is not obtained.
Therefore, we propose in this work a suitable description
of the dominant collision mechanisms which allows to extend
the formal derivation of the spherical
harmonic expansion (SHE) model given in \cite{nbafmnegulescu}.
Then a new energy--transport (ET) model is formally derived as
diffusive limit from this SHE model \cite{nbadegond}.

Numerical discretization of classical ET equations
has already been studied in many papers~:
by using mixed finite elements schemes e.g. in \cite{etnumeriq,gadau,gadau2,holst,Lab,marrocco},
ENO schemes in \cite{jerome}, finite difference methods \cite{fournie,ringho}
and finite volume schemes in \cite{chainais}.
In \cite{etnumeriq,holst}, the authors propose a drift-diffusion reformulation
which allows to use an accurate Sharfetter-Gummel scheme with exponential 
fitting \cite{expfitting} and, moreover, to decouple the ET model.
However, in this quantum case, the involved form of the diffusion coefficients
does not allow for a decoupled drift-diffusion reformulation.
Then we will use a more traditional approach with mixed finite elements which can be directly applied
since the obtained ET model turns out to be in symmetric form.
Then, a Gummel type algorithm is used as outer iterations of the solution 
of the coupled energy-transport-Schr\"odinger-Poisson model, 
and the (non-linear) ET discrete system is solved 
by means of a Newton scheme.
Moreover, passing to the limit in the energy relaxation term, a subband 
drift-diffusion equation is recovered in the form of \cite{ddsp} 
with a more accurate description of the diffusion coefficients 
taking into account the collisional mechanisms (see also \cite{ddspnum} 
for numerical simulations).
This work is then an extension of \cite{ddspnum} 
to a more general collisional framework.

The outline of the paper is the following. In Section
\ref{deriv} we set the assumptions on the
collision mechanisms and we briefly present the formal derivation 
of the SHE model.
Then, we derive from this latter model the novel
subband ET model. Finally, a subband
drift-diffusion equation is obtained, as limit when the relaxation
time goes to infinity. Section \ref{simul} is devoted to the
numerical issues. Subsection \ref{simul1} presents the complete
stationary model, the mixed finite elements scheme is described in
Subsection \ref{simul2}, and the iterative approach is outlined in
Subsection \ref{simul3}. Numerical simulations of an ultra-scaled
Double-Gate MOSFET are presented in Subsection \ref{simul4}.

\Section{Formal derivation of adiabatic fluid-quantum models}\label{deriv}

\subsection{The quantum-kinetic framework}
We will assume in this work that the confinement direction is one
dimensional whereas the transport takes place in a two dimensional
domain. The domain is denoted $\Omega=\omega\times [0,\ell]$ with
$\omega\subset \RR^2$. The first two directions, called $x\in
\omega$, correspond to the classical description of the gas, whereas
in the third direction $z\in [0,\ell]$ quantum effects occur. The
quantum confinement of the electron gas is described thanks to the
eigen-elements of the 1D Schr\"odinger operator. They are denoted
$(\Epsilon_n,\chi_n)_{n\in \NN^*}$ and solve the eigenvalue
problem~: \beq \left\{
\begin{array}{ll}
\ds -\frac{\hbar^2}{2} \,\frac{d}{dz}\left(\frac{1}{m^*(z)}\frac{d}{dz}
  \chi_n\right)+ U \chi_n=\Epsilon_n\chi_n,  \\
\ds \chi_n(x,\cdot)\in H^1_0(0,\ell),\quad
\ds \int_0^\ell \chi_n\,\chi_{n'}\,dz=\delta_{nn'}\,.
\end{array}
\right.
\label{sdm}
\end{equation}
In this equation $\hbar$ is the reduced Planck constant, $m^*$
the effective mass.
It is known that the eigenvalues $\Epsilon_n$ in \fref{sdm} form an increasing
sequence tending to $+\infty$.
%, and $U_c$ is a given potential barrier between the silicon and the oxide
%in the nanotransistor.
These functions depend on the potential energy defined by $U=-eV$,
where $e$ is the elementary charge and $V$ denotes the
self-consistent electrostatic potential, solution of the Poisson
equation \beq \div_{x,z}(\varepsilon_R(x,z) \nabla_{x,z}
V)=\frac{e}{\varepsilon_0}(N_e-N_D). \label{poisson}
\end{equation}
Here $\varepsilon_R(x,z)$ denotes the relative permittivity,
$\varepsilon_0$ the permittivity constant in vacuum, $N_D(x,z)$ is
the prescribed doping density and $N_e(t,x,z)$ is the electron density.
This density is
described by a sequence of distribution functions
$(f_n)_{n\in\NN^*}$ describing the repartition on each subband for
the classical direction $x\in\omega$ and the corresponding momentum
variable $k\in \RR^2$. It is written as
$$
N_e(t,x,z)=\sum_{n=1}^{+\infty} \left(\int_{\RR^2}
f_n(t,x,k)\,dk\right) |\chi_n|^2(t,x,z).
$$

The evolution of distribution functions is governed by classical
transport model in the $x$ direction parallel to the gas. The total
energy of the $n$th subband is defined by
\beq\label{energy}
\eps_n(t,x,k) = \frac{|k|^2\hbar^2}{2 m^*} +
\Epsilon_n(t,x). \eeq
Therefore the energy-band diagram of the
semiconductor crystal is spherically symmetric and strictly monotone
with respect to $|k|$. Then the Brillouin zone (which represents the
elementary cell of the dual lattice $L^*$) is equal to $\RR^2$.
Moreover, we point out that, in contrast with the classical counterpart, the
energy-band depends on space and time.
In a kinetic collisional framework, the distribution function $f_n$ of
the $n$th subband satisfies the rescaled Boltzmann transport
equation \cite{nbadegondgenieys,ddspnum}~: \beq\label{BTEintro}
\alpha^2\pa_t f_n^\alpha +\alpha(\nabla_k \eps_n\cdot\nabla_x
f_n^\alpha - \nabla_x \eps_n\cdot\nabla_k f_n^\alpha) =
Q_{ld}(f^\alpha)_n + \frac{\alpha^2}{\beta} Q_e(f^\alpha)_n, \eeq
where $Q_{ld}$ is the collision operator for the lattice defect
collisions and $Q_e$ is the collision operator for the elastic, non
linear electron-electron collisions, accounting for intra--band scattering
as well as for transitions between subbands. $\alpha$ and $\beta$ are
dimensionless parameters that satisfy $\alpha \ll \beta \ll 1$.

The main classes of lattice-defects that we shall consider
are impurities and phonons \cite{ashmer}~:
$$
Q_{ld}(f) = Q_{imp}(f) +Q_{ph}(f).
$$
The elastic character of the impurity scattering leads to
$$
Q_{imp}(f)_n(k):= \sum_{n'\in \NN^*} \int_{\RR^2} \Phi^{imp}_{n,n'}(k,k')
\delta (\eps_n(k) - \eps_{n'}(k')) (f_{n'}(t,x,k') -f_n(t,x,k)) \,dk',
$$
where $\delta$ is the Dirac measure and the
dependence on $t$, $x$ of $\eps$ and $\Phi^{imp}$ has been
omitted. The cross-section is assumed to be symmetric:
$\Phi^{imp}_{n,n'}(k,k')=\Phi^{imp}_{n',n}(k',k)$.

\noindent The electron-phonon collision operator is considered as
$$
\begin{array}{l}
\ds Q_{ph}(f)_n(k) =  \sum_{n'\in \NN^*} \int_{\RR^2} \Phi^{ph}_{n,n'}(k,k') \\[2mm]
\ds ([(N_{ph}+1) \delta(\eps_n(k) - \eps_{n'}(k')+\alpha^2\eps_{ph})
+ N_{ph} \delta(\eps_n(k) - \eps_{n'}(k')-\alpha^2\eps_{ph})]f_{n'}(k')(1-\eta f_n(k)) - \\[2mm]
\ds [(N_{ph}+1) \delta(\eps_{n'}(k') - \eps_n(k)+\alpha^2\eps_{ph})
+ N_{ph} \delta(\eps_{n'}(k') - \eps_n(k)-\alpha^2\eps_{ph})]f_n(k)(1-\eta f_{n'}(k')))dk',
\end{array}
$$
where again, $\Phi_{n,n'}^{ph}(k,k')=\Phi_{n',n}^{ph}(k',k)$, $\eps_{ph}$ is
the phonon energy,
$\eta\geq 0$ is a dimensionless distribution function scaling factor
and the terms $0\leq 1-\eta f_n\leq 1$ express the Pauli exclusion principle.
$N_{ph}$ is the phonon occupation number, given by the Bose-Einstein statistics
\beq\label{Nph}
N_{ph} = \left(e^{\alpha^2 \eps_{ph}/(k_B T_L)} -1 \right)^{-1},
\eeq
with $T_L$ the lattice temperature and $k_B$ the Boltzmann constant.
Formally expanding the phonon collision operator in power of $\alpha^2$, we get
$$
Q_{ph}^\alpha(f)=Q_{ph,0}(f)+\alpha^2 Q_{ph,1}^\alpha(f)
$$
where $Q_{ph,1}^\alpha$ is of order 1 when $\alpha$ goes to 0.

The electron-electron collision operator is given by \cite{reggiani}
\beq\label{Qe}
\begin{array}{l}
\ds Q_e(f)_n(k) = \\
\quad \ds \sum_{n',r,s} \int_{(\RR^2)^3} \Phi^e_{n,n',r,s}(k,k',k_1,k_1')
\delta(\eps_n + \eps_{n',1} - \eps'_r-\eps'_{s,1}) \delta(k+k_1-k'-k'_1)
     \\[2mm]
\quad \ds [f'_r f'_{s,1} (1-\eta f_n)(1-\eta f_{n',1})
-f_n f_{n',1} (1-\eta f'_r)(1-\eta f'_{s,1})]\,dk'dk_1dk'_1.
\end{array}
\eeq The notation $f_{n',1}$, $f'_r$ and $f'_{s,1}$ stands for
$f_{n'}(k_1)$, $f_r(k')$ and $f_s(k'_1)$, respectively.

We define then the elastic collision operator
\beq\label{Q0}
\begin{array}{ll}
\ds Q_0(f)_n & = \ds Q_{imp}(f)_n + Q_{ph,0}(f)_n \\[2mm]
& = \ds \sum_{n'\in \NN^*}\int_{\RR^2} \Phi^0_{n,n'} (k,k')
\delta(\eps_n(k)-\eps_{n'}(k'))(f_{n'}(k')-f_n(k)) \,dk',
\end{array}
\eeq
where $\Phi^0_{n,n'} = \Phi^{imp}_{n,n'} + (2 N_{ph}+1) \Phi^{ph}_{n,n'}$.
We set
\beq\label{Q1alpha}
Q_1^\alpha(f) = Q_{ph,1}^\alpha(f) + \frac{1}{\beta} Q_e(f).
\eeq

Then, the kinetic equation, starting point for the
diffusive limits, is written in the following form
 \beq\label{BTE}
\alpha^2\pa_t f_n^\alpha +\alpha(\nabla_k \eps_n\cdot\nabla_x
f_n^\alpha - \nabla_x \eps_n\cdot\nabla_k f_n^\alpha) =
Q_0(f^\alpha)_n + \alpha^2 Q_1^\alpha(f^\alpha)_n. \eeq

\subsection{Definitions and notations}

We first recall the coarea formula : for any $C^1$ function $g : B
\mapsto \RR$, and any test function $\psi\in C^0(B)$, we have~:
$$
\int_B \psi(k)\,dk = \int_\RR \left(
  \int_{g^{-1}(\eps)} \psi(k) \frac{dS_\eps(k)}{|\nabla g(k)|} \right) \,d\eps,
$$
where $dS_\eps(k)$ denotes the Euclidian surface element on the
manifold $g^{-1}(\eps)$. We denote $dN_\eps(k) = dS_\eps(k)/|\nabla
g(k)|$. Taking $g(k)=|k|^2\hbar^2/(2m^*)$, the set of
possible wave vectors of electrons belonging to the $n-th$ subband
and having total energy $\eps$ is given by
$S_{\eps-\epsilon_n}=\{k\in \RR^2$ s. t.
$|k|^2=2m^*\hbar^{-2}(\eps-\Epsilon_n)\}$ and
$dN_{\eps-\epsilon_n}=\frac{dS_{\eps-\epsilon_n}}{|k|\hbar^2/m^*}$
where $dS_{\eps-\epsilon_n}$ is the surface measure of the ball
$S_{\eps-\epsilon_n}$. The coarea formula leads to~:
\beq\label{coarea} \sum_{n\in \NN^*}\int_{\RR^2}
\psi_n(k)\,dk = \sum_{n\in\NN^*}\int_{\epsilon_n}^{+\infty} \left(
  \int_{S_{\eps-\epsilon_n}} \psi_n(k) dN_{\eps-\epsilon_n}(k) \right)\,d\eps,
\eeq
and
$$
\sum_{n\in \NN^*}\int_{\RR^2}
\psi_n(k)\delta(\Epsilon_n+\frac{|k|^2\hbar^2}{2m^*}-\eps)\,dk =
\sum_{n\in \NN^*} \int_{S_{\eps-\epsilon_n}} \psi_n(k) dN_{\eps-\epsilon_n}(k).
$$

\begin{definition}\label{definition}
We will use the following notations :
\begin{itemize}
\item The density of states is defined by :
$$
N(t,x,\eps):=\sum_{n\in \NN^*} \int_{S_{\eps-\epsilon_n}}
dN_{\eps-\epsilon_n}(k) = 2\pi\frac{m^*}{\hbar^2} \calN(t,x,\eps),
$$
where $\calN(t,x,\eps)=\max \{ n\in \NN^* \ / \ \Epsilon_n(t,x)\leq
\eps\}$, with the convention of $\calN(t,x,\eps)=0$ if $\eps < \Epsilon_1(t,x)$.
\item The Fermi-Dirac function is given by
$$
\calF_{\mu,T}(t,x,\eps) =
\left(\eta+\exp\left(\frac{\eps-\mu}{k_B T}\right)\right)^{-1}.
$$
\item We introduce the two Hilbert spaces
$$
\LL^2 := \{f=(f_n)_{n\in \NN^*},
\quad \sum_{n=1}^{+\infty} \int_{\RR^2} |f_n(k)|^2\,dk < +\infty \},
$$
endowed with the natural scalar product
$$
\langle f,g \rangle = \sum_{n\in \NN^*}\int_{\RR^2} f_n(k) g_n(k) \,dk ,
$$
and %the Hilbert space for the Fermi-Dirac function
$$
\LL^2_\calF = \{ f\in L^2(\RR) \mbox{ s. t. } \int_\RR f^2(\eps)
\frac{d\eps}{\calF(\eps)(1-\eta \calF(\eps))} <+\infty \},
$$
endowed with the weighted scalar product defined
by
$$
\langle f,g \rangle_{\calF} = \int_\RR f(\eps) g(\eps)
\frac{d\eps}{\calF(\eps)(1-\eta \calF(\eps))}.
$$
\end{itemize}
\end{definition}

We will make the following assumption on the cross-section~:
\begin{assumption}\label{hypPhi0}
The coefficient $\Phi^0_{n,n'}$ satisfies, for $\lambda_0$ and
$\lambda_1$ two positive constants,`
$$
0<\lambda_0 < \Phi^0_{n,n'} N(t,x,\eps_n) < \lambda_1, \quad
\Phi^0_{n,n'}(k,k') = \Phi^0_{n',n}(k',k),
$$
where $N$ is the density of state defined above.
\end{assumption}

\subsection{First macroscopic scaling : the spherical harmonic expansion model}

For the sake of completness of this work, we present in this section
the limit $\alpha\to 0$ of the kinetic equation \fref{BTE}.
All calculations will be done formally and we refer the reader to
\cite{nbafmnegulescu} where the rigorous derivation is studied.
We consider the Hilbert expansion
$$
f^\alpha = f^0 + \alpha f^1 + \alpha^2 f^2 + \cdots
$$
By linearity of the operator $Q_0$ and by identifying the term of equal
powers of $\alpha$ in \fref{BTE}, we obtain
\begin{eqnarray}
\ds &&Q_0(f^0) = 0, \label{ordrealpha0}\\
\ds &&Q_0(f^1)_n= \frac{\hbar^2 k}{m^*}\cdot \nabla_x f_n^0
-\nabla_x\Epsilon_n\cdot \nabla_k f_n^0,
\label{ordrealpha1}\\
\ds &&Q_0(f^2)_n=\pa_t f_n^0 +\frac{\hbar^2 k}{m^*}\cdot\nabla_x
f_n^1 -\nabla_x\Epsilon_n\cdot \nabla_k f_n^1-Q_1^0(f^0)_n, \label{ordrealpha2}
\end{eqnarray}
where $Q_1^0$ is obtained by taking $\alpha = 0$ in the expression \fref{Q1alpha}.

We will then make use of the following properties
of the collision operator (see \cite{nbafmnegulescu}).
\begin{proposition}\label{propQ0}
Under Assumption \ref{hypPhi0}, the elastic collision operator $Q_0$
defined in \fref{Q0} satisfies the following properties~:
\begin{enumerate}
\item The linear operator $Q_0~: \LL^2\mapsto \LL^2$ is a bounded,
  symmetric, non-positive operator.
\item For any bounded function $\psi~: \RR\mapsto \RR$, we denote
  $\psi(\eps)_n(k)=\psi(\frac{|k|^2\hbar^2}{2m^*}+\Epsilon_n)$. Then,
$$
\forall\,f\in \LL^2,\ Q_0(\psi(\eps)f)=\psi(\eps) Q_0(f).
$$
\item The Kernel of $Q_0$ is the set
$$
\mbox{Ker } Q_0 = \{ f\in \LL^2, \quad \mbox{ s.t. } \exists\,
\psi:\RR\to \RR, \quad f = \psi(\eps) \}.
$$
\item The range $R(Q_0)$ is closed and coincide with the orthogonal
  of the kernel of $Q_0$ given by~:
$$
(\mbox{Ker } Q_0)^\bot = \{ f\in \LL^2, \quad \mbox{s.t.} \sum_{n\in
\NN^*} \int_{S_{\eps-\epsilon_n}} f_n(k) \,dN_{\eps-\epsilon_n}(k)
=0, \mbox{ for a.e. } \eps\geq \Epsilon_1 \}.
$$
\end{enumerate}
\end{proposition}

From Proposition \ref{propQ0} and \fref{ordrealpha0},
we deduce that $f^0$ is an energy dependent function~:
$$
f^0_n(t,x,k) = F(t,x,\eps_n).
$$
Choosing $\psi:\RR\to\RR$ such that $k\psi(\eps)\in \LL^2$, we deduce from
Proposition \ref{propQ0} that there
exists a unique solution $\xi$ in (Ker $Q_0)^\bot$ such that
$$
-Q_0(\xi) = \frac{\hbar^2 k}{m^*} \psi(\eps).
$$
We can write $\xi=\Theta\cdot\psi(\eps)$ and from the second item of
Proposition \ref{propQ0}, $\Theta$ is independent of the choice of
the function $\psi$. Then, equation \fref{ordrealpha1} leads to
\beq\label{f1} f^1_n(t,x,k) = -\Theta_n(t,x,k)\cdot \nabla_x
F(t,x,\eps_n). \eeq Finally, the solvability condition of equation
\fref{ordrealpha2} is that the right hand side belongs to (Ker
$Q_0)^\bot$. This leads to \beq\label{step1she} \sum_{n\in \NN^*}
\int_{S_{\eps-\epsilon_n}} (\pa_t f_n^0 + \nabla_k\eps_n\cdot
\nabla_x f_n^1 - \nabla_x \Epsilon_n \cdot \nabla_k f^1_n -
Q_1^0(f^0)_n)\,dN_{\eps-\epsilon_n}(k) = 0, \mbox{
for a.e. } \eps\geq \Epsilon_1. \eeq Let us denote \beq\label{SeS1}
S_e(F) = \sum_{n\in \NN^*} \int_{S_{\eps-\epsilon_n}}
Q_e(F)_n\,dN_{\eps-\epsilon_n}(k), \quad \mbox{ and
} \quad S_1(F) = \sum_{n\in \NN^*} \int_{S_{\eps-\epsilon_n}}
Q_{ph,1}^0 (F)_n\,dN_{\eps-\epsilon_n}(k). \eeq
Multiplying \fref{step1she} by an energy-dependent test function
$\phi(\eps)$ and integrating with respect to the variable $\eps$, we
obtain for the first term~:
$$
\begin{array}{c}
\ds \int_{\epsilon_1}^\infty \sum_{n\in \NN^*}
\int_{S_{\eps-\epsilon_n}} \pa_t f_n^0 \,dN_{\eps-\epsilon_n}(k)
\phi(\eps) \,d\eps = \int_{\epsilon_1}^\infty
\sum_{n\in \NN^*} \int_{S_{\eps-\epsilon_n}} (\pa_t
F + \pa_\eps F \pa_t \Epsilon_n) \,dN_{\eps-\epsilon_n}(k)
\phi(\eps) \,d\eps   \\[2mm]
\ds = \int_{\epsilon_1}^\infty N \pa_t F \phi(\eps)\,d\eps +
\int_{\epsilon_1}^\infty \pa_\eps F \left(
\sum_{n\in \NN^*} \pa_t \Epsilon_n
\int_{S_{\eps-\epsilon_n}} dN_{\eps-\epsilon_n}(k)\right) \phi(\eps)
\,d\eps.
\end{array}
$$
Using the coarea formula \fref{coarea}, we deduce that
$$
\begin{array}{c}
\ds \int_{\epsilon_1}^\infty \left(\sum_{n\in \NN^*}
\int_{S_{\eps-\epsilon_n}} (\nabla_k\eps_n\cdot \nabla_x f_n^1-
\nabla_x\Epsilon_n\cdot\nabla_k f_n^1)
\,dN_{\eps-\epsilon_n}(k)\right)
\phi(\eps) \,d\eps = \\
\ds =\sum_{n\in \NN^*}\int_{\RR^2}
(\nabla_x\cdot(\frac{\hbar^2 k}{m^*} f_n^1)- \nabla_k\cdot (f_n^1
\nabla_x \Epsilon_n)) \phi(\eps_n) \,dk = \nabla_x
\cdot \left(\sum_{n\in \NN^*}\int_{\RR^2}
\frac{\hbar^2 k}{m^*} f_n^1
\phi(\eps_n)\,dk \right) =\\
\ds = -\int_{\epsilon_1}^\infty \nabla_x\cdot\left(
\sum_{n\in \NN^*} \int_{S_{\eps-\epsilon_n}}
\frac{\hbar^2 k}{m^*}\otimes\Theta_n \,dN_{\eps-\epsilon_n}(k) \cdot
\nabla_x F\right) \phi(\eps) \,d\eps,
\end{array}
$$
where the last identity is a consequence of \fref{f1}. We define the
diffusion matrix by \beq\label{D} D(t,x,\eps) := \sum_{n\in
\NN^*}\int_{S_{\eps-\epsilon_n}} \frac{\hbar^2 k}{m^*}
\otimes\Theta_n \,dN_{\eps-\epsilon_n}(k), \eeq and the current
density by \beq\label{J} J(t,x,\eps) = - D(t,x,\eps)\cdot \nabla_x
F(t,x,\eps). \eeq With these notations, we get that in the
distributional sense, equation \fref{step1she} is equivalent to the
spherical harmonic expansion (SHE) model \beq\label{SHE} N\pa_t F +
\nabla_x \cdot J - \kappa\pa_\eps F = \frac{1}{\beta} S_e(F) +
S_1(F), \eeq where $\kappa$ is given by \beq\label{kappa}
\kappa(t,x,\eps) = -2\pi\frac{m^*}{\hbar^2} \pa_t\left(\sum_{n\in
\NN^*}(\eps-\Epsilon_n)^+\right). \eeq The notation $u^+=\max
\{0,u\}$ denotes the positive part of $u$. We recall moreover a
property of the diffusion matrix $D$ stated in Lemma 2.8 of
\cite{nbafmnegulescu}. We point out that the effect
of the confinement is reflected in the special form of the
coefficients of \fref{SHE}, which involve the subband energies.

\begin{lemma}\label{Dpos}
The diffusion matrix $D(t,x,\eps)$ define in \fref{D} is a symmetric
and nonnegative $2\times 2$ matrix.
\end{lemma}

\begin{remark}
We end this section with a particular choice of the
cross--section $\Phi^0_{n,n'}$, which allows to compute explicitly
the diffusion matrix. If the cross-section is an energy-dependent
function of the form
$$
\Phi^0_{n,n'}(t,x,k,k') = \Phi^0(t,x,\eps_n),
$$
then, after a straightforward computation, we have that
$$
\Theta_n(t,x,k) = \frac{1}{\Phi^0(t,x,\eps_n)N(t,x,\eps_n)}\, \nabla_k \eps_n.
$$
Therefore, the diffusion matrix defined in \fref{D} has the expression
\beq\label{Dexplicit}
D(t,x,\eps) = \frac{\pi\hbar^2/m^*}{\Phi^0(t,x,\eps)N(t,x,\eps)}
\sum_{n\in \NN^*} (\eps - \Epsilon_n)^+ \,Id.
\eeq
\end{remark}

\subsection{Second macroscopic scaling :
the energy-transport model}

We start from the SHE model \fref{SHE} and we assume that the
electron-electron collision operator is dominant with respect to the
second order correction of the phonon collision operator and
therefore $\beta\ll 1$, in order to obtain an ET model.
%In this section we retrieve
%the energy-transport model that has been derived directly from the
%Boltzmann equation \fref{BTEintro} in \cite{nbafmnegulescu}.
Passing through the SHE model, instead of starting directly from the
Boltzmann equation allows to get an explicit expression of the
coefficients, which is needed for numerical purpose. Moreover, the
considered dominant scattering mechanisms provide an energy
relaxation term in the macroscopic limiting model.

The formal limit $\beta\to 0$ in \fref{SHE} is again performed by
means of a Hilbert expansion
$$
F = F^0 + \beta F^1 + \cdots
$$
Identifying equal powers of $\beta$ implies \beq\label{ordre1}
S_e(F_0)= 0, \eeq \beq\label{ordre2} N \pa_t F^0 + \nabla_x J^0 -
\kappa \pa_\eps F^0 - S_1(F^0) = D_{F^0} S_e (F^1), \eeq where
$D_{F^0} S_e$ denotes the Fr\'echet derivative of $S_e$ at $F^0$.

We summarize below some useful properties of the collision operator
$S_e$ defined in \fref{SeS1} and of its Fr\'echet derivative.
\begin{proposition}\label{propSe}
Under micro-reversibility assumptions on the cross-section $\Phi^e$, 
the operator $S_e$ satisfies the following properties~:

(i) For all $f,g\in L^2(\RR)$, we have :
\begin{eqnarray*}
%\begin{array}{l}
\ds &&\int_{\RR} S_e(f)(\eps) g(\eps) d\eps =\\[2mm]
&=& -\frac 14 \sum_{n,n',r,s} \int_{(\RR^2)^4} \Phi^e_{n,n',r,s}
\delta_\eps \delta_k
[f(\eps_r(k'))f(\eps_s(k'_1)) (1-\eta f(\eps_n(k)))(1-\eta f(\eps_{n'}(k_1)))  \\[2mm]
\ds  &-& f(\eps_n(k)) f(\eps_{n'}(k_1))(1-\eta f(\eps_r(k')))(1-\eta
f(\eps_s(k'_1))) [g'_r + g'_{s,1} -g_n-g_{n',1}]\, dkdk_1dk'dk'_1.
\end{eqnarray*}
(ii) kernel :
$$
\mbox{ Ker } S_e = \{ f\in L^2(\RR) ; \exists\, \mu(t,x), T(t,x)
\mbox{ s. t. } f(t,x,\eps)= \calF_{\mu,T} (t,x,\eps)\},
$$
where $\calF_{\mu,T}$ is the so-called Fermi-Dirac distribution function
(see Definition \ref{definition}).
\end{proposition}

\begin{proposition}\label{propDSe}
The linear operator $D_\calF S_e$ satisfies

(i) $D_\calF S_e$ is bounded, symmetric, non-positive on $\LL^2_{\calF}$.

(ii) The kernel of $D_\calF S_e$ is given by
$$
\mbox{ Ker } (D_\calF S_e) = Span \{\calF(1-\eta\calF), \calF(1-\eta\calF)\eps\}
$$

(iii) The range of $D_\calF S_e$ is closed and we have
$$
R(D_\calF S_e) = \mbox{ Ker }(D_\calF S_e) ^\bot =
\left\{ f\in \LL^2_\calF \quad / \quad \int_\RR f(\eps)
\left(\begin{array}{c} 1 \\ \eps \end{array}\right)
\,d\eps = 0 \right\}.
$$
\end{proposition}

These properties are an easy consequence of Proposition 3.16,
Proposition 3.17 and Proposition 3.19 of \cite{nbafmnegulescu}, using the fact that :
$$
\begin{array}{ll}
\ds \int_{\RR} S_e(f)(\eps) g(\eps) \,d\eps  & = \ds
\int_{\epsilon_1}^{+\infty} \sum_{n\in \NN^*}
\int_{S_{\eps-\epsilon_n}}
Q_e(f)_n(\eps)g(\eps) \,dN_{\eps-\epsilon_n}(k) d\eps \\[2mm]
& \ds = \sum_{n\in \NN^*}
\int_{\epsilon_n}^{+\infty} \int_{S_{\eps-\epsilon_n}}
Q_e(f)_n(\eps)g(\eps) \,dN_{\eps-\epsilon_n}(k) d\eps \\[2mm]
& \ds = \sum_{n\in\NN^*} \int_{\RR^2} Q_e(f)_n(\eps_n)g(\eps_n) \, dk,
\end{array}
$$
where we use the coarea formula \fref{coarea} for the last identity.

\medskip

\noindent{\bf Formal derivation of energy-transport model.} 
Let us come back to the formal limit $\beta\to 0$ in \eqref{SHE}.
Thanks to Proposition \ref{propSe}, equation \fref{ordre1} implies
that there exist $\mu(t,x)$ and $T(t,x)$ such that 
\beq\label{F0}
F^0(t,x,\eps) = \calF_{\mu,T} (t,x,\eps) \eeq 
From Proposition
\ref{propDSe} we deduce that equation \fref{ordre2} admits a
solution iff \beq\label{solvability} \int_{\RR} (N \pa_t F^0 +
\nabla_x\cdot J^0 - \kappa \pa_\eps F^0- S_1(F^0))
\left(\begin{array}{c} 1 \\ \eps \end{array}\right) \,d\eps = 0.
\eeq For the first term, the definition of the density of states $N$
(see Definition \ref{definition}) implies that
$N(t,x,\eps)=2\pi\frac{m^*}{\hbar^2} n$ if $\eps\in\,
[\Epsilon_n,\Epsilon_{n+1})$ and vanishes for $\eps<\Epsilon_1$.
Then
$$
\begin{array}{ll}
\ds \int_\RR N \pa_t F^0 \left(\begin{array}{c} 1 \\ \eps
\end{array} \right)  \,d\eps & \ds = \sum_{n\in \NN^*}
\int_{\epsilon_n}^{\epsilon_{n+1}} 2\pi\frac{m^*}{\hbar^2} n\pa_t
F_0
\left(\begin{array}{c} 1 \\ \eps \end{array}\right) \,d\eps  \\
& \ds = \pa_t \left(\int_\RR N  F^0 \left(\begin{array}{c} 1 \\ \eps \end{array}
\right)  \,d\eps \right) - 2\pi\frac{m^*}{\hbar^2} \sum_{n\in \NN^*}
\pa_t\Epsilon_n F^0(\Epsilon_n)
\left(\begin{array}{c} 1 \\ \Epsilon_n \end{array} \right).
\end{array}
$$
Using the expression of the current \fref{J}, we can rewrite the
second term of \fref{solvability}~:
$$
\int_\RR \nabla_x\cdot J^0 \left(\begin{array}{c} 1 \\
\eps \end{array}\right)\,d\eps
= - \nabla_x \cdot\left[\int_\RR D(t,x,\eps) \cdot \nabla_x F^0(t,x,\eps)
\left(\begin{array}{c} 1 \\ \eps \end{array}\right) \,d\eps \right].
$$
From \fref{kappa}, we deduce
$$
\begin{array}{ll}
\ds \int_\RR \kappa \pa_\eps F^0 & \left(\begin{array}{c} 1 \\
 \eps \end{array}\right) \,d\eps   \ds= \sum_{n\in \NN^*}
 \int_{\epsilon_n}^\infty
2\pi \frac{m^*}{\hbar^2}\pa_t\Epsilon_n \, \pa_\eps F^0
\left(\begin{array}{c} 1 \\ \eps \end{array}\right) \,d\eps \\[2mm]
&\ds = -\sum_{n\in \NN^*} 2\pi
\frac{m^*}{\hbar^2}\pa_t\Epsilon_n \, F^0(\Epsilon_n)
\left(\begin{array}{c} 1 \\ \Epsilon_n
\end{array}\right) - \sum_{n\in \NN^*} 2\pi \frac{m^*}{\hbar^2}\pa_t\Epsilon_n
\int_{\epsilon_n}^\infty F^0(\eps) \left(\begin{array}{c} 0 \\ 1
\end{array}\right)\,d\eps,
\end{array}
$$
where we use an integration by part for the last identity.
Finally, the solvability condition \fref{solvability} writes in
the following form~:
\beq\label{EnTr}
\begin{array}{c}
\ds \pa_t \left(\int_\RR N  F^0 \left(\begin{array}{c} 1 \\
\eps \end{array}\right)  \,d\eps \right)+ \sum_{n\in
\NN^*} 2\pi \frac{m^*}{\hbar^2} \pa_t\Epsilon_n
\int_{\epsilon_n}^\infty
F^0(\eps) \left(\begin{array}{c} 0 \\ 1 \end{array}\right)\,d\eps \\[2mm]
\ds -\nabla_x\cdot\left[\int_\RR D(t,x,\eps) \cdot \nabla_x F^0(t,x,\eps)
\left(\begin{array}{c} 1 \\ \eps \end{array}\right) \,d\eps \right]
= \int_\RR S_1(F^0) \left(\begin{array}{c} 1 \\ \eps \end{array}\right)\,d\eps.
\end{array}
\eeq

% Using the coarea formula \fref{coarea} we can rewrite :
% $$
% \begin{array}{c}
% \ds \sum_k \int_{\RR^2}
% \pa_t f_k^0 \left(\begin{array}{c} 1 \\ \eps_k \end{array}\right)
% dv  +
% \sum_k\int_{\RR^2} \nabla_x \cdot (vf_k^1)
% \left(\begin{array}{c} 1 \\ \eps_k \end{array}\right)\,dv \\[2mm]
% \ds + \sum_k\int_{\RR^2} \nabla_v \cdot (\nabla_x\Epsilon_k f_k^1)
% \left(\begin{array}{c} 1 \\ \eps_k \end{array}\right)\,dv =0.
% \end{array}
% $$
% Since we have $\nabla_v \eps_k = v$, we deduce after some integrations
% by part that
% $$
% \sum_k \int_{\RR^2}
% \pa_t f_k^0 \left(\begin{array}{c} 1 \\ \eps_k \end{array}\right)
% dv  +
% \nabla_x \cdot \left( \sum_k \int_{\RR^2} vf_k^1
% \left(\begin{array}{c} 1 \\ \eps_k \end{array}\right)\,dv\right) =0.
% $$
% For the first term, we have~:
% $$
% \begin{array}{c}
% \ds \sum_k \int_{\RR^2} \pa_t f_k^0 \,dv =
% \frac{d}{dt} \left( \sum_k \int_{\RR^2}
% \calF_{\mu,T}(t,x,\eps_k)\,dv \right).
% \end{array}
% $$
% And for the second equation, we obtain
% $$
% \begin{array}{c}
% \ds \sum_k \int_{\RR^2} \pa_t f_k^0 \eps_k \,dv =
% \frac{d}{dt} \left( \sum_k \int_{\RR^2}
% \eps_k \calF_{\mu,T}(t,x,\eps_k)\,dv \right) -
% \sum_k \pa_t \Epsilon_k \int_{\RR^2}\calF_{\mu,T}(t,x,\eps_k)\,dv .
% \end{array}
% $$
Let us denote by $\rho$ and $\rho\calE$ the charge density and the
energy density, respectively, associated to the Fermi-Dirac
distribution function $\calF_{\mu,T}$~:
\beq\label{rho}
\rho_{\mu,T}(t,x):= \int_\RR N \calF_{\mu,T}(t,x,\eps) \,d\eps =
\sum_{n\in \NN^*}\int_{\RR^2}\calF_{\mu,T}(t,x,\eps_n)\,dk, \eeq \beq\label{calE}
\rho\calE_{\mu,T}(t,x) := \int_\RR N \calF_{\mu,T}(t,x,\eps) \eps
\,d\eps = \sum_{n\in \NN^*} \int_{\RR^2} \eps_n
\calF_{\mu,T}(t,x,\eps_n)\,dk.
\eeq
We can state easily that for a Fermi-Dirac function, we have
$$
\nabla_x \calF_{\mu,T}(t,x,\eps) = -\calF_{\mu,T}(1-\eta\calF_{\mu,T})
\left(\eps \nabla_x(\frac{1}{k_BT}) - \nabla_x (\frac{\mu}{k_BT})\right).
$$
Then, equation \fref{EnTr} reads, using $\int_\RR S_1(\calF)\,d\eps=0$,
\begin{eqnarray}
\label{ET1}
&\ds \pa_t \rho_{\mu,T} - \nabla_x \cdot J_1 = 0, \\[2mm]
\label{ET2} &\ds \pa_t (\rho\calE_{\mu,T}) + \sum_{n
\in \NN^*} 2\pi\frac{m^*}{\hbar^2}\pa_t\Epsilon_n
\int_{\epsilon_n}^\infty \calF_{\mu,T}(t,x,\eps) \,d\eps - \nabla_x
\cdot J_2 = W,
%:= \int_\RR \eps S_1(\calF_{\mu,T}) \,d\eps.
\end{eqnarray}
where we denote
\begin{eqnarray}
\label{J1}
&&\ds J_1 := \int_\RR D(t,x,\eps) \cdot \nabla_x \calF_{\mu,T}(t,x,\eps) \,d\eps
= \DD_{00} \nabla_x(\frac{\mu}{k_BT}) - \DD_{01}\nabla_x (\frac{1}{k_BT}), \\[2mm]
\label{J2}
&&\ds J_2 := \int_\RR \eps\, D(t,x,\eps) \cdot \nabla_x \calF_{\mu,T}(t,x,\eps)
 \,d\eps = \DD_{10} \nabla_x(\frac{\mu}{k_BT}) - \DD_{11}\nabla_x
(\frac{1}{k_BT}),
\end{eqnarray}
and where the diffusion coefficients are defined by
\beq\label{diffcoeff}
\DD_{ij}(t,x)=\int_\RR D(t,x,\eps) \eps^{i+j}
\calF_{\mu,T} (1- \eta\calF_{\mu,T})\,d\eps, \qquad \mbox{ for }\ i,
j = 0,1,
\eeq
with $D$ being defined in \fref{D}. The system
\fref{ET1}--\fref{diffcoeff} forms the energy-transport model in
the transport direction for a partially confined electron gas. We
recover the general form for energy-transport models for
semiconductors (see \cite{nbadegond,nbadegondgenieys,jungel} and
references therein). The right hand side of the energy equation
\fref{ET2} is the so-called relaxation term $W$ and it is defined by
\beq\label{relaxW}
W := \int_\RR \eps S_1(\calF_{\mu,T}) \,d\eps
=\sum_{n\in \NN^*}\int_{\epsilon_n}^{+\infty}
\left(\int_{S_{\eps-\epsilon_n}} Q^0_{ph,1}(f)_n
dN_{\eps-\epsilon_n}(k)\,\right) \eps \,d\eps. \eeq

As for the SHE model \fref{SHE}, the presence of
the subband energies in the diffusion matrix and in the relaxation
term reflects the effect of the confinement in the transport
equation.

Then we have the following important property for the diffusion matrix
which is an easy consequence of expression \fref{diffcoeff} and Lemma \ref{Dpos}.
\begin{lemma}\label{DDpositif}
The diffusion matrix
$$
\calD := \begin{pmatrix}
\DD_{00} & \DD_{01} \\
\DD_{10} & \DD_{11}
\end{pmatrix}
$$
is a symmetric and positive definite matrix.
\end{lemma}

\begin{remark}
After straightforward calculations, we can have an explicit expression
of $\rho$ and $\rho\calE$. In fact,
$$
\rho_{\mu,T}(t,x)=\sum_{n\in
\NN^*}\int_{\epsilon_n}^{\epsilon_{n+1}} 2\pi \frac{m^*}{\hbar^2}n
\calF_{\mu,T}(t,x,\eps)\,d\eps = \frac{2\pi m^* k_B T}{\eta\hbar^2}
\sum_{n\in \NN^*}\log(1+\eta
\exp(\frac{\mu-\Epsilon_n}{k_B T})),
$$
$$
\rho\calE_{\mu,T}(t,x) = \sum_{n\in \NN^*} 2\pi
\frac{m^*}{\hbar^2} n \int_{\epsilon_n}^{\epsilon_{n+1}}  \eps \,
\calF_{\mu,T}(t,x,\eps) \,d\eps
$$
if $\eta >0$. For Boltzmann statistics $\eta=0$, we have
\beq\label{rhoexpl}
\rho_{\mu,T}(t,x)= \sum_{n\in\NN^*} 2\pi\frac{m^*}{\hbar^2} k_B T(t,x)
\exp (\frac{\mu(t,x)-\Epsilon_n(t,x)}{k_B T(t,x)}),
\eeq
\beq\label{Eexpl} \rho\calE_{\mu,T}(t,x) =
\sum_{n\in \NN^*} 2\pi\frac{m^*}{\hbar^2} k_B
T(t,x) (\Epsilon_n(t,x) + T(t,x))
\exp(\frac{\mu(t,x)-\Epsilon_n(t,x)}{k_B T(t,x)}). \eeq
% And the currents are given by
% \begin{eqnarray}
% \label{J1bolstat}
% &\ds J_1(t,x) = \int_\RR D(t,x,\eps) \nabla_x(e^{(\mu(t,x)-\eps)/(k_B T(t,x))})
% \,d\eps \ ;  \\
% \label{J2bolstat}
% & \ds J_2 (t,x)=
% \int_\RR D(t,x,\eps)\, \eps\,\nabla_x (e^{(\mu(t,x)-\eps)/(k_B T(t,x))})\,d\eps.
% \end{eqnarray}
\end{remark}

% \begin{remark}
% By directly deriving this energy-transport model from a diffusion
% approximation of the Boltzmann equation \fref{BTE} in which the leading
% order term is $Q_0(f)+Q_e(f)$, we obtain a similar model with a more
% complicated expression for the diffusion matrix $\DD_{ij}$ than
% \fref{diffcoeff} (see \cite{nbafmnegulescu}).
% For a numerical treatment of this model, it is more affordable to use
% the model obtained in this work for which we can have an explicit
% expression of the diffusion matrix.
% Actually, in a simplified case where the cross-section $\Phi^0$
% is assumed to be an energy-dependent function, the diffusion matrix
% $D$ has the expression \fref{Dexplicit}.
%
% \end{remark}

\subsection{Relaxation term W}

The relaxation term defined in \fref{relaxW} measures the
influence of the interaction of phonons with the charge carriers.
In \cite{ddspnum} we have formally established that a diffusion
limit of the kinetic Boltzmann transport equation coupled to
subband model in the scaling of dominant phonon-electron interaction
leads to a drift-diffusion system in the transport direction
coupled to the subband model.
Using the coarea formula we can rewrite \fref{relaxW} as
$$
W=\sum_{n\in \NN^*} \int_{\RR^2} \eps_n
Q^0_{ph,1}(\calF_{\mu,T})\,dk.
$$
Moreover, we have
$$
\begin{array}{l}
\ds \sum_{n\in \NN^*} \int_{\RR^2} \eps_n
Q^\alpha_{ph}(\calF_{\mu,T})_n\,dk =
\sum_{n,n'} \int_{\RR^4} \Phi^{ph}_{n,n'}(k,k') \\[2mm]
\qquad\ds ([(N_{ph}+1) \delta(\eps_n - \eps'_{n'}+\alpha^2\eps_{ph})
+ N_{ph} \delta(\eps_n - \eps'_{n'}-\alpha^2\eps_{ph})]
\eps_n \calF'_{n'}(1-\eta \calF_n) - \\[2mm]
\qquad\ds [(N_{ph}+1) \delta(\eps'_{n'} - \eps_n+\alpha^2\eps_{ph}) +
N_{ph} \delta(\eps'_{n'} - \eps_n-\alpha^2\eps_{ph})]
\eps_n \calF_n(1-\eta \calF'_{n'}))\,dkdk',
\end{array}
$$
where the notation $\eps'_{n'}$, $\calF_n$ and $\calF'_{n'}$ stands,
respectively, for $\eps_{n'}(k')$, $\calF_{\mu,T}(\eps_n)$ and
$\calF_{\mu,T}(\eps'_{n'})$. Since the Fermi-Dirac distribution function is
energy-dependent, we have
$$
\begin{array}{l}
\ds \sum_{n\in \NN^*} \int_{\RR^2} \eps_n
Q^\alpha_{ph}(\calF_{\mu,T})_n\,dk = \sum_{n,n'} \int_{\RR^4}
\Phi^{ph}_{n,n'}(k,k')
(1-\eta\calF'_{n'})(1-\eta\calF_n) \\[2mm]
\qquad\ds ([(N_{ph}+1) \delta(\eps_n - \eps'_{n'}+\alpha^2\eps_{ph})
+ N_{ph} \delta(\eps_n - \eps'_{n'}-\alpha^2\eps_{ph})]
\eps_n M'_{n'} - \\[2mm]
\qquad\ds [(N_{ph}+1) \delta(\eps'_{n'} - \eps_n+\alpha^2\eps_{ph}) +
N_{ph} \delta(\eps'_{n'} - \eps_n-\alpha^2\eps_{ph})]
\eps_n M_n)\,dkdk',
\end{array}
$$
where $M_n=e^{(\mu-\eps_n)/(k_B T)}$ is the Maxwellian and it
satisfies $\calF_n=M_n(1-\eta \calF_n)$. Moreover,
$$
\begin{array}{l}
\ds \sum_{n\in \NN^*} \int_{\RR^2} \eps_n
Q^\alpha_{ph}(\calF_{\mu,T})_n\,dk = \sum_{n,n'} \int_{\RR^4}
\Phi^{ph}_{n,n'}(k,k')
(1-\eta\calF'_{n'})(1-\eta\calF_n) \\[2mm]
\qquad\ds (\delta(\eps_n - \eps'_{n'}+\alpha^2\eps_{ph})
[(N_{ph}+1) M'_{n'}- N_{ph} M_n] \eps_n - \\[2mm]
\qquad\ds \delta(\eps'_{n'} - \eps_n+\alpha^2\eps_{ph})
[(N_{ph}+1) M_n- N_{ph}M'_{n'}] (\eps'_{n'}+\alpha^2\eps_{ph})) \,dkdk'.
\end{array}
$$
By changing the variable $\eps_n$ with $\eps'_{n'}$ in the first term
of the sum, we notice that
only the $\alpha^2$-factor term does not vanish in this last
identity, which can be rewritten as
$$
\begin{array}{l}
\ds \sum_{n,n'} \int_{\RR^4} \Phi^{ph}_{n,n'}
(1-\eta\calF'_{n'})(1-\eta\calF_n) \eps_{ph} \delta(\eps_n -
\eps'_{n'} -\alpha^2\eps_{ph})
[N_{ph}M'_{n'}-(N_{ph}+1) M_n]\,dkdk' =  \\[2mm]
\ds \sum_{n,n'} \int_{\RR^4} \Phi^{ph}_{n,n'}
(1-\eta\calF'_{n'})(1-\eta\calF_n) \eps_{ph} \delta(\eps_n -
\eps'_{n'} -\alpha^2\eps_{ph}) N_{ph} M_n (e^{\frac{\alpha^2
\eps_{ph}}{k_B T}}- e^{\frac{\alpha^2\eps_{ph}}{k_B T_L}})\,dkdk',
\end{array}
$$
where the phonon occupation number $N_{ph}$ is defined in \eqref{Nph}. 
Letting $\alpha\to 0$, we have that
$N_{ph}(e^{\alpha^2\eps_{ph}/k_BT}-e^{\alpha^2\eps_{ph}/k_BT_L})\to
T_L(\frac{1}{T} - \frac{1}{T_L})$. Thus 
\beq\label{Wrelax1} W=
T_L(\frac{1}{T} - \frac{1}{T_L}) \sum_{n,m} \int_{\RR^4}
\Phi^{ph}_{n,n'}(k,k') \calF_n (1-\eta\calF_n) \eps_{ph}
\delta(\eps_n - \eps_{n'}')\,dkdk'. \eeq

The following lemma proves that $W$ is a temperature relaxation term
which relaxes $T$ to the lattice temperature $T_L$.
\begin{lemma}\label{temprelax}
Let $W$ be defined in \fref{relaxW}. Then, we have
$$
W\cdot (T-T_L) \leq 0.
$$
\end{lemma}
\begin{proof}
The proof of this result is an immediate consequence of
\fref{Wrelax1}.
\end{proof}

\subsection{Formal derivation of drift-diffusion equation. }
\label{ddsection}
In the case where the electron-phonon scattering is dominant, after a
rescaling we have $W=\frac{1}{\gamma} \widetilde{W}$ with a
parameter $\gamma\ll 1$. Then equation \fref{ET2} with expression
\fref{Wrelax1} implies that formally in the limit
$\gamma \rightarrow 0$, we have $T=T_L$. Then \fref{ET1} leads to
the well-known drift-diffusion model \beq\label{dd} \pa_t\rho_{\mu}
- \nabla_x \left(\frac{\DD_{00}}{k_B T_L} \nabla_x \mu\right) = 0.
\eeq Moreover, assuming $\eta=0$, the equilibrium is then given by a
Boltzmann statistics. We deduce therefore from \fref{rhoexpl} that
for $T=T_L$
\begin{equation}\label{rho:dd}
\rho_\mu = 2\pi \frac{m^*}{\hbar^2} k_BT_L e^{\mu/k_BT_L}
\sum_{n\in \NN^*} e^{-\epsilon_n/k_BT_L}.
\end{equation}
Then
$$
\nabla_x\rho_\mu = \frac{1}{k_BT_L} \rho_\mu \nabla_x \mu + \rho_\mu
\frac{\nabla_x (\sum_n e^{-\epsilon_n/k_BT_L})}{\sum_n e^{-\epsilon_n/k_BT_L}}.
$$
We can introduce as in \cite{ddsp} the effective potential energy defined by
\beq\label{Vs}
V_s =-k_B T_L\log \big(\sum_{n\in
\NN^*} e^{-\epsilon_n/k_BT_L}\big). \eeq
Then, we have
$$
\rho_\mu \nabla_x\mu = \left( k_B T_L\nabla_x\rho_\mu + \rho_\mu \nabla_xV_s\right),
$$
such that, denoting $\DD=\DD_{00}/\rho_\mu$, we recover from
\fref{dd} the standard formulation of the drift-diffusion system
\cite{ddsp}~:
\beq\label{DD} \pa_t\rho_\mu - \nabla_x\cdot (\DD
(k_B T_L\nabla_x \rho_\mu + \rho_\mu\nabla_xV_s))=0. \eeq

\begin{remark}
We conclude the section noticing the similarities and the
differences between the classical and the partially confined
energy-transport model. Due to partial confinement in the $z$ direction,
the electron density in the transport direction contains $T$ as factor
rather than $T^{3/2}$. Moreover, we point out that
the system \fref{ET1}-\fref{diffcoeff} in the variables $\mu/k_bT$,
$-1/k_b T$ is in symmetric form, with the electric forces appearing
in the diffusion coefficients through the eigenenergies
$\Epsilon_n$ (which in turns depend on the electrostatic potential). 
In the partially confined framework, the diffusive
limiting process brought directly to a set of variables which can be
interpreted as dual entropy variables (using the denomination of
nonequilibrium thermodynamics \cite{DeGroot,Kreuzer}). 
In this framework, $T$ can be
interpreted as electron temperature, and the variable $\mu$ as
quasi-Fermi potential energy. This fact is clear considering the
drift-diffusion model \eqref{DD}. Indeed, using the effective
potential energy $V_s$ defined in \fref{Vs}, we can write the
electron density \fref{rho:dd} as
$$ \rho_{\mu}= 2\pi \frac{m^*}{\hbar^2} k_BT_L
e^{(\mu-V_s)/k_BT_L},$$  which gives the classical relation between
electron density and quasi-Fermi energy. A chemical potential can
then be defined as $\mu_{chem}=\mu-V_s$.
\end{remark}

\Section{Numerical simulation}\label{simul}

The device we are modelling in this work is a nanoscale Double-Gate
MOSFET (Metal Oxide Semiconductor Field Effect Transistor) such as in \cite{ddspnum}.
This device consists of a silicon film, characterized by two highly doped
regions near the Ohmic contacts (denoted by source and drain) and
an active region,  called channel, with lower doping.
The silicon film is sandwiched between two thin layers of silicon dioxide
$SiO_2$, each of them with a gate contact.

\begin{figure}[htbp]
\begin{center}
\includegraphics*[width=7.7cm]{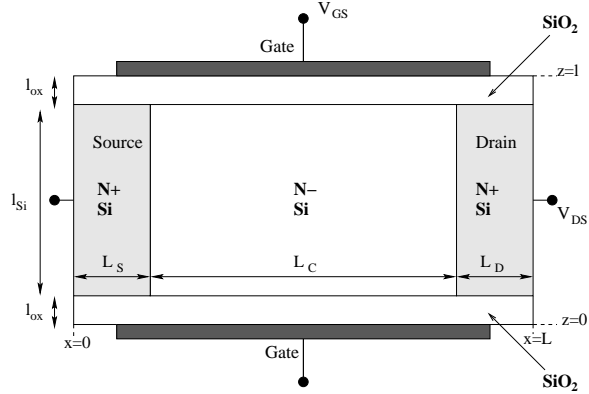}
\end{center}
\caption{\label{dispo} Schematic representation of the modeled
  device. }
\end{figure}

We assume invariance in the $x_2$ direction
(infinite boundary conditions), so that the problem is studied in a
$(x_1,z)$-domain. The device occupies a region of a 2D domain denoted by
$\Omega$=$[0,L]\times[0,\ell]$. A schematic representation of the device is
shown in Figure \ref{dispo}.

\subsection{Energy-transport--Schr\"odinger--Poisson system}
\label{simul1}

In the following, we describe the collisional transport in the Double-Gate MOSFET,
schematized in Figure \ref{dispo}, by means of the energy-transport model
\fref{ET1}--\fref{diffcoeff}. The confinement is described by the
subband decomposition approach, which involves the resolution of the
eigenvalue problem \fref{sdm}, taking also into
account the presence of the oxide. Moreover, in order to provide
explicitly computable diffusion and relaxation terms, the following
physical assumptions will be used.

\begin{assumption}\label{hypo}

\medskip
\begin{itemize}
\item {\bf (H1)} The cross-sections $\Phi^0$ and $\Phi^{ph}$
are assumed to be energy dependent functions and to
have the following expression (see \cite{nbadegond,nbademasc,etnumeriq})
\begin{equation}\label{h1}
\Phi^0_{n,n'}(t,x,k,k') = \phi^0(t,x)
\eps_n^{\bet}, \qquad \Phi^{ph}_{n,n'}(t,x,k,k') =
\phi^{ph}(t,x) \eps_n^\bet,
\end{equation}
with $-2<\bet<2$. In the physical literature, the values $\bet=0$
and $\bet=1/2$ have been used \cite{chenetal,jungel,lyumkis}.

\item {\bf (H2)} The electron density and the energy are assumed to be given
  by non-degenerate Boltzmann statistics, i.e. $\eta = 0$, as in \fref{rhoexpl}, \fref{Eexpl}.

\end{itemize}
\end{assumption}

Using Assumption (H1), we deduce that the diffusion matrix
\fref{Dexplicit} has the following expression \beq\label{Dexpl}
D(t,x,\eps)=\frac{1}{\phi^0(t,x)\eps^\bet \calN(\eps)}
\sum_{n\in \NN^*} (\eps-\Epsilon_n)^+ \,Id, \eeq
where $\calN(\eps)=\max\{n\in \NN^* \ / \ \Epsilon_n \leq \eps\}$ is
the number of non-zero terms in the sum (see Definition \ref{definition}).
Moreover, thanks to the coarea formula in \fref{Wrelax1}, we have
$$
W = -4\pi^2\eps_{ph}(1-\frac{T_L}{T}) \int_{\epsilon_1}^{+\infty}
\phi^{ph}\, \eps^\bet \,\calN^2(\eps)
\calF(\eps)(1-\eta\calF(\eps))\,d\eps,
$$
which under Assumption (H2), reads
%we have for Boltzmann
%statistics ($\eta=0$),
$$
W = - 4\pi^2 \eps_{ph} (1-\frac{T_L}{T}) e^{\mu/(k_B T)}
\int_{\epsilon_1}^{+\infty} \phi^{ph}\, \eps^\bet \,\calN^2(\eps)
e^{-\eps/(k_B T)} \,d\eps.
$$
By defining
 \beq\label{W0} W_0 = 4\pi^2 \eps_{ph}
\int_{\epsilon_1}^{+\infty} \phi^{ph}\, \eps^\bet \,\calN^2(\eps)
e^{-\eps/(k_B T)} \,d\eps, \eeq
we have the compact expression
 \beq\label{WW0}
W = -W_0\left(1-\frac{T_L}{T}\right)e^{\mu/(k_B T)}.
\eeq

Assumption (H2) implies also that the density of
charge carriers $N_e(t,x,z)$ is given by
$$
N_e=\sum_{n\in \NN^*} \int_{\RR^2} \calF_{\mu,T}(t,x,
\frac{|k|^2\hbar^2}{2m^*}+\Epsilon_n)\,dk |\chi_n|^2
=\frac{2\pi m^* k_BT}{\hbar^2} \sum_{n\in \NN^*} e^{(\mu-\epsilon_n)/k_BT} |\chi_n|^2.
$$

Finally, the coupled subband energy-transport model
under Assumption \ref{hypo} is given by~: Find
$\mu(t,x),~ T(t,x),~(\Epsilon_n(t,x),\chi_n(t,x))$ for $n \geq 1$,
and $V(t,x,z)$ such that
\begin{eqnarray}
&\ds \pa_t \rho_{\mu,T} - \nabla_x \cdot J_1 = 0,\qquad \mbox{in } (0,L)
\label{eq:density}   \\[2mm]
&\ds \pa_t (\rho\calE_{\mu,T}) + \sum_{n \in \NN^*}
2\pi\frac{m^*}{\hbar^2} \pa_t\Epsilon_n k_B T(t,x) e^{(\mu-\epsilon_n)/k_BT}
- \nabla_x \cdot J_2 = W,~~\mbox{in } (0,L)
\label{eq:energy}    \\[2mm]
&\ds \left\{
\begin{array}{ll}
\ds -\frac{\hbar^2}{2} \,\frac{d}{dz}\left(\frac{1}{m^*(z)}\frac{d}{dz}
  \chi_n\right)- e(V+V_c) \chi_n=\Epsilon_n\chi_n, \qquad \mbox{in } (0,\ell) \\
\ds \chi_n(t,x,\cdot)\in H^1_0(0,\ell),\quad
\ds \int_0^\ell \chi_n\,\chi_{n'}\,dz=\delta_{nn'}\,,
\end{array}
\right.
\label{eq:schrodinger}  \\[2mm]
&\ds \mbox{div}_{x,z}(\eps_R\nabla_{x,z} V) = \frac {e}{\eps_0}
\left(\frac{2\pi m^* k_BT}{\hbar^2}
\sum_{n\in \NN^*} e^{(\mu-\epsilon_n)/k_BT}
|\chi_n|^2- N_D\right), \quad \mbox{in } \Omega,
\label{eq:poisson}
\end{eqnarray}
where the expressions of $\rho_{\mu,T}$ and $\calE_{\mu,T}$ with respect to the unknows
are given in \eqref{rhoexpl}--\eqref{Eexpl}.
In \fref{eq:schrodinger} the effective mass $m^*$
takes different values in the $Si$  and in the $SiO_2$ domain.
Moreover, $V_c$ represents a given potential barrier between the
silicon and the oxide. The currents $J_1$
and $J_2$ are given by the expressions
\begin{eqnarray}
&&\ds J_1= \DD_{00} \nabla_x\big(\frac{\mu}{k_BT}\big) - \DD_{01}
\nabla_x\big(\frac{1}{k_BT}\big), \label{eq:J1}\\[2mm]
&&\ds J_2= \DD_{10} \nabla_x\big(\frac{\mu}{k_BT}\big) - \DD_{11}
\nabla_x\big(\frac{1}{k_BT}\big), \label{eq:J2}
\end{eqnarray}
where, under Assumption \ref{hypo}, the diffusion coefficients are given by
$$
\DD_{ij} = \frac{1}{\phi^0} \sum_{n\in\NN^*}\int_{\epsilon_n}^{+\infty}
\frac{\eps^{i+j-\bet}(\eps-\Epsilon_n)} {\calN(\eps)}
\,e^{(\mu-\eps)/k_BT }\,d\eps.
$$
The relaxation term $W$ is given by \fref{WW0}.

This system is complemented with initial and boundary conditions.
In particular, at the ohmic contacts and at the gate (see Figure \ref{dispo}), we will impose
Dirichlet boundary conditions for the potential, otherwise we fix
homogeneous Neumann boundary conditions, which model isolating
conditions.
\begin{eqnarray}\label{eqbord:V} V(x,z) &= {V_{Gate}},
\quad &\mbox{ for } z\in\{0,\ell\}, x\in\mbox{Gate}\ ; \\
V(x,z) &= V_D, \quad &\mbox{ for } x\in\{0,L\},
z\in(0,\ell);\\
\frac{\pa V}{\pa \nu} &= 0, \quad &\mbox{ elsewhere},
\end{eqnarray}
where $\nu$ is the outward unit normal. Since the transport
occurs only in the longitudinal direction, we just have to
impose boundary conditions in $x=0$ and
$x=L$ for $\mu$ and $T$. The temperature is assumed to be at the
lattice temperature $T_L$,
thus
\beq\label{eqbord:temp} T(x) = T_L, \quad \mbox{ for } x\in \{0,L\}.
\eeq

Then, we consider that the surface density of the charge carriers is
almost constant near the frontiers $x=0$ and $x=L$ and given by
$N_s^b$. The surface density being the integral over $z$ of the
total density ($N_s^b=N^{+} \times \ell_{Si}$), we
deduce \beq\label{eqbord:mu} \mu(x)=\mu_b := k_B T_L\log
\left(\frac{N_s^b \hbar^2}{2\pi m k_B T_L \sum_n
e^{-\epsilon_n/k_BT_L}}\right), \quad \mbox{ for } x\in \{0,L\}.
\eeq

\subsection{Stationary system}
\label{simul2}

Let us introduce the notations
\beq\label{nota:uv}
u=\frac{\mu}{k_BT}, \quad v = -\frac{1}{k_BT}.
\eeq
Then we can rewrite the expressions of the current \fref{eq:J1}--\fref{eq:J2} as
\begin{eqnarray}
\label{J1DD}
&&\ds J_1= \DD_{00}(u,v)\nabla_x u +\DD_{01}(u,v)\nabla_x v,
\\[2mm]
\label{J2DD}
&&\ds J_2= \DD_{10}(u,v)\nabla_x u +\DD_{11}(u,v)\nabla_x v,
\end{eqnarray}
where the diffusion coefficients are given by
\beq\label{diffmat}
\DD_{ij}(u,v) = \frac{1}{\phi^0} \sum_{n\in
\NN^*}\int_{\epsilon_n}^{+\infty}
\frac{\eps^{i+j-\bet}(\eps-\Epsilon_n)} {\calN(\eps)} \,e^{u+\eps
v}\,d\eps.
\eeq
We define the relaxation coefficient in the same way~:
\beq\label{W0uv} W_0(u,v) = 4\pi^2 \eps_{ph}
\int_{\epsilon_1}^{+\infty} \phi^{ph}\, \eps^\bet \,\calN^2(\eps)
e^{(\eps-\mu) v} \,d\eps. \eeq

Then, the stationary version of the energy-transport subband system
\fref{eq:density}--\fref{eq:energy} in variable $u$ and $v$ writes
\begin{eqnarray}
\label{eq1:ETSP}
&\ds -\nabla_x(\DD_{00}(u,v)\nabla_x u +\DD_{01}(u,v)\nabla_x v) = 0, \\[2mm]
\label{eq2:ETSP}
&\ds -\nabla_x(\DD_{10}(u,v)\nabla_x u +\DD_{11}(u,v)\nabla_x v) =
-W_0(u,v)\left(1+k_B T_L v\right),
% \\[2mm]
% \label{sch:ETSP}
% &\ds -\frac 12\pa^2_{z} \chi_n + V \chi_n = \Epsilon_n \chi_n,\qquad n\in\NN^*  \\[2mm]
% &\ds \chi_n(0)=\chi_n(1)=0, \quad \int_0^1 |\chi_n|^2\,dz = 1, \\[2mm]
% \label{pois:ETSP}
% &\ds -\Delta_{x,z} V = -\frac{2\pi m}{\hbar} \frac{e^u}{v} \sum_n e^{\epsilon_n v} |\chi_n|^2
\end{eqnarray}
Boundary conditions \fref{eqbord:temp}--\fref{eqbord:mu} become
$$
v(x) = v_b := -\frac{1}{k_BT_L}, \quad \mbox{ for } x\in \{0,L\},
$$
and
$$
u(x)=u_b := \log \left(\frac{N_s^b \hbar^2}{2\pi m
k_B T_L \sum_n e^{-\epsilon_n/k_BT_L}}\right), \quad \mbox{ for }
x\in \{0,L\}.
$$

\subsection{Numerical approach for the energy-transport system}
\label{simul3}

We introduce a partition of $[0,L]$ with nodes $x_i$, $i=0,\cdots,N_x$, and
a partition of $[0,\ell]$ with nodes $z_j$, $j=0,\cdots,N_z$.
We assume that the partitions are uniform and denote $h=x_i-x_{i-1}$.
Then, we mesh the domain $[0,L]\times[0,\ell]$ with rectangular triangles
using the nodes ($x_i,z_j$) previously defined.
The Schr\"odinger equations and the
Poisson equation are discretized with conforming $P^1$ finite elements.

\medskip
We consider here in details the discretization
scheme for the equations governing $u$ and $v$, assuming first that
the eigenenergies $\Epsilon_n$ are known. Using the following notations
$$
U=(u,v)^\top, \quad \mathcal J=(J_1,J_2)^\top , \quad
W(U) =(0,W_0(u,v)(1+k_B T_L v))^\top
$$
equations \eqref{J1DD}-\eqref{J2DD} and \eqref{eq1:ETSP}-\eqref{eq2:ETSP}
can be written
in compact form as
\begin{equation}\label{system}
 \mathcal J= \DD(U) \nabla_x U, \quad \nabla_x \cdot \mathcal J =W(U).
\end{equation}
Denoting by $U_i$  an approximation of $U(x_i)$,
we take the piecewise constant
approximation of $U$ given, in the interval $I_i:=(x_{i-1},x_i)$, by
$$
\overline{U}_i =\frac {U_{i-1}+U_i}{2}.
$$
and define the piecewise constant diffusion coefficients and relaxation terms as
\begin{equation}\label{eqnumD00}
\overline \DD_{kl}=\DD_{kl}(\overline U), \quad \mbox{ for }
k,l=0,1; \qquad \overline W = W(\overline U).
\end{equation}

%We denote $u_i$ (resp. $v_i$) an approximation of $u(x_i)$ (resp. $v(x_i)$).
%In the interval $(x_{i-1},x_i)$, we take the piecewise constant
%approximation of $u$ and $v$ given by
%$$
%\overline{u}_i =\frac {u_{i-1}+u_i}{2}, \quad
%\overline{v}_i =\frac {v_{i-1}+v_i}{2},
%$$
%and define the piecewise constant diffusion and relaxation coefficients as
%\begin{equation}\label{eqnumD00}
%\DD_{kl}^i=\DD_{kl}(\overline u_i,\overline v_i), \quad \mbox{ for } k=0,1,\
%l= 0,1;
%\qquad W_0^i = W_0(\overline u_i,\overline v_i).
%\end{equation}
We are going to use a mixed finite element discretization of lowest
order in hybridized form \cite{brezzifortin} (see also \cite{gadau,gadau2}
for applications to ET).
Let us introduce the following finite dimensional spaces~:
\begin{eqnarray*}
  X_{h} &=& \{ \sigma\in L^2(\Omega)\,:\, \sigma(x)=a_i+b_i (x-x_{B_i})
             \mbox{ in } I_i, \;i=1,\ldots,N_x \}, \\
  Y_{h} &=& \{ \xi\in L^2(\Omega)\,:\, \xi\mbox{ is constant in }
             I_i,\;i=1,\ldots,N_x \}, \\
  Z_{{h},\chi} &=& \{ q\mbox{ is defined at the nodes }x_0,\ldots,
             x_{N_x},\, q(x_0) = \chi(0), \;q(x_{N_x}) = \chi(1) \},
\end{eqnarray*}
where $x_{B_i}$ denotes the central point of the interval $I_i$, and $\chi$ is  prescribed.

Then, the mixed-hybrid formulation of \eqref{system}
reads as follows:
Find $\mathcal J_h\in X_h^2$, $P_h\in Y_h^2$,
and $U_h\in Z_{h,u_b} \times Z_{h,v_b}$ such that
\begin{eqnarray}
  \sum_{i=1}^{N_x}{\left(\int_{I_i}{{\phi}_h\cdot{{ \DD(\overline U_i)}^{-1}}\mathcal J_h} dx+
    \int_{I_i}{P_h\cdot\nabla_x{\phi}_h} dx -
    [{\phi}_h\cdot U_h]_{x_{i-1}}^{x_i}\right)} & = & 0
    \label{flux1},\\
  \sum_{i=1}^{N_x}{\left(\int_{I_i}{{\Psi}_h\cdot\nabla_x\mathcal J_h} dx -
    \int_{I_i}{W(\overline U_i) \cdot {\Psi}_h} dx \right)}& = & 0
    \label{flux2},\\
  -\sum_{i=1}^{N_x}{[{\mu}_h\cdot \mathcal J_h]_{x_{i-1}}^{x_i}} & = & 0
    \label{flux3}
\end{eqnarray}
for all ${\phi}_h\in {X}_h^2$, ${\Psi}_h\in {Y}_h^2$, and
${\mu}_h\in {Z}_{h,0}^2$. Equation \eqref{flux1} is derived from the
weak formulation of the first equations in \eqref{system};
\eqref{flux2} comes from the weak form of the second equations in
\eqref{system}; and finally, \eqref{flux3} imposes the continuity of
the currents at the nodes.
%{\bf $P_h$ and $U_h$ are approximation of $U$.}

Thanks to the discontinuity of the spaces $X_h$ and $Y_h$, we can
apply static condensation in order to reduce the size of the discrete 
system and obtain an algebraic system for the variable $U_h$ only.
More precisely, choosing first the local basis
$$\phi_h=\left\{\begin{array}{ll}(1,0)^\top
        & \mbox{in } I_i\\
        (0,0)^\top & \mbox{elsewhere}
      \end{array}\right., \
\phi_h=\left\{\begin{array}{ll}(0,1)^\top
        & \mbox{in } I_i\\
        (0,0)^\top & \mbox{elsewhere}
      \end{array}\right.
$$
in \eqref{flux1} and then, analogously choosing $\Psi_h$ in \eqref{flux2},
 we obtain the
piecewise linear (discrete) current
\begin{equation}\label{eqnumcurr}
\mathcal J_{h|{I_i}} = \DD(\overline U_i) ~ \frac{ U_i - U_{i-1}}
{h} + W( \overline U_i) (x-x_{B_i}).
\end{equation}
Imposing continuity at the nodes (through \eqref{flux3}) we obtain
the final system
\begin{equation}\label{eqnumU}
-\DD(\overline U_i) U_{i-1} + (\DD(\overline U_i)+\DD(\overline U_{i+1})) U_{i} -
\DD(\overline U_{i+1})U_{i+1} = -\frac{h^2} {2} (W(\overline U_i) + W(\overline U_{i+1})),
\end{equation}
for $i=1,\ldots,N_x-1$. We point out explicitly that, since the
first component of $W(U)$ is null, the approximation of the current
$J_1$ is piecewise constant (see \eqref{eqnumcurr}) and that, thanks
to \eqref{flux3},
 it is indeed globally constant.

%An approximation of the current $J_1$ \fref{J1DD}
%in the interval $(x_{i-1},x_i)$ is given by
%\begin{equation}\label{eqnumJ}
%J^i_1= \frac 1 h \DD_{00}^i (u_i - u_{i-1}) + \frac 1 h \DD_{01}^i(v_i - v_{i-1})
%\end{equation}
%This approximation corresponds to use a finite element
%method of lowest order, with the piecewise constant
%approximation \fref{eqnumD00} for the diffusion coefficients.
%Imposing continuity in the nodes (that is, $J_1^i=J_1^{i+1}$ due to
%\fref{eq1:ETSP}) gives the scheme
%\begin{equation}\label{eqnum1}
%-\DD_{00}^i u_{i-1} + (\DD_{00}^i+\DD_{00}^{i+1}) u_{i} - \DD_{00}^{i+1} u_{i+1}
%-\DD_{01}^i v_{i-1} + (\DD_{01}^i+\DD_{01}^{i+1}) v_{i} - \DD_{01}^{i+1} v_{i+1} = 0.
%\end{equation}
%The (globally constant) current is then retreived by \fref{eqnumJ}.
%By using the same approximation which consists of a discretization
%in $P^1$ finite elements of $(u,v)$ with piecewise constant approximation
%of the diffusion coefficients and of the relaxation coefficient $W_0$,
%we deduce the following discretization of equation \fref{eq2:ETSP}
%\begin{equation}\label{eqnum2}
%\begin{array}{c}
%\ds -\DD_{10}^i u_{i-1} + (\DD_{10}^i+\DD_{10}^{i+1})u_{i}- \DD_{10}^{i+1} u_{i+1}
%-\DD_{11}^i v_{i-1} +(\DD_{11}^i+\DD_{11}^{i+1}) v_{i}- \DD_{11}^{i+1}v_{i+1} \\[2mm]
%\ds + \frac h2(W_0^i+W_0^{i+1}) + h k_BT_L\left( v_{i-1}\frac{W_0^i}6
%+ v_i \frac{W_0^i+W_0^{i+1}}3 + v_{i+1}\frac{W_0^{i+1}}6\right) = 0.
%\end{array}
%\end{equation}

%Then a discretization of \fref{eq1:ETSP}--\fref{eq2:ETSP} is given by
%\fref{eqnum1}--\fref{eqnum2}.
System \eqref{eqnumU} forms a non-linear system in the unknown $(u,v)$ that
can be solved using a Newton algorithm.
We point out that the Jacobian corresponding to this non-linear system can be
easily computed noticing that, from the expressions \fref{eqnumD00}
and \fref{diffmat}, we have
$$
\frac{\pa \overline \DD_{k\ell}^i}{\pa u_i} =
\frac{\pa \overline \DD_{k\ell}^i}{\pa u_{i-1}} =
\frac 12 \overline \DD_{k\ell}^i, \qquad \frac{\pa
\overline \DD_{k\ell}^i}{\pa v_i} = \frac{\pa
\overline \DD_{k\ell}^i}{\pa v_{i-1}} = \frac 12
\overline \DD_{k\ell+1}^i,
$$
and that similar relations hold for the partial derivatives of $\bar{W}^i$.

\begin{remark}\label{DDexpl}
The expression \fref{diffmat} is not practical for numerical purpose.
However, with the definition $\calN(\eps)=\max\{n\in \NN^*~: \Epsilon_n\leq \eps\}$ , we have
$$
\begin{array}{ll}
\ds \DD_{ij}(u,v) & \ds = \frac{1}{\phi^0}
\sum_{n=1}^{+\infty} \sum_{m=n}^{+\infty} \int_{\epsilon_m}^{\epsilon_{m+1}}
\frac{(\eps-\Epsilon_n) \eps^{i+j-\bet}}{m} e^{u+\eps v}\,d\eps \\[2mm]
&\ds = \frac{1}{\phi^0}\sum_{m=1}^{+\infty}
\int_{\epsilon_m}^{\epsilon_{m+1}} (\eps-\frac{\sum_{n=1}^m
\Epsilon_n}{m}) \eps^{i+j-\bet} e^{u+\eps v}\,d\eps,
\end{array}
$$
by interchanging the sums over $m$ and $n$. We can rewrite
$$
\begin{array}{ll}
\ds \DD_{ij}(u,v) = &\ds \frac{1}{\phi^0} \int_{\epsilon_1}^{+\infty}
\eps^{i+j+1-\bet} e^{u+\eps v}\,d\eps \\[2mm]
&\ds + \frac{1}{\phi^0}\sum_{m=1}^{+\infty}
\left(\frac{\sum_{n=1}^{m-1} \Epsilon_n}{m-1}-\frac{\sum_{n=1}^{m}
\Epsilon_n}{m}\right) \int_{\epsilon_m}^{+\infty} \eps^{i+j-\bet}
e^{u+\eps v}\,d\eps,
\end{array}
$$
with the convention that $\sum_{n=1}^{m-1} \Epsilon_n/(m-1) = 0$ for
$m=1$. Then, in the actual numerical computation,
we can get an accurate approximation of $\DD_{ij}$ by truncating the
infinite sum to a finite number of eigenmodes. In fact,
$(\Epsilon_n)_n$ forms an increasing sequence going to $+\infty$,
thus, since $v<0$, we have that $(e^{\epsilon_n v})_n$ fast
decreases to $0$.
\end{remark}

\subsection{Algorithm}
\label{simul4}

We are now ready to describe the algorithm used
for the numerical resolution of the stationary
subband energy-transport model. The first step of
the algorithm is the computation of the thermal equilibrium
solution, with no applied drain-source bias. In this case the
temperature and the Fermi level are constant along the device,
therefore the problem reduces to solving the Poisson equation
\fref{eq:poisson} for a given temperature and Fermi level computed
thanks to the boundary conditions. 
The computed potential at thermal equilibrium is used as a starting
data for
%Once we have computed the potential at the thermal equilibrium, we solve the whole system thanks to 
the following Gummel \cite{gummel} iteration process~:
\begin{enumerate}
\item Lets $V_{old}$ be a given potential.
\item We solve the eigenvalue problem \fref{eq:schrodinger} on each vertical
  slice of the domain by diagonalization of the Hamiltonian.
  Therefore we obtain the set $\{\chi_n(x_i,z_j)\}$ and
  $\{\Epsilon_n(x_i)\}$.
\item We implement the Newton procedure which has been described above
  for the computation of $(u,v)$.
\item We compute the density of charge carriers corresponding
to the right hand side of \fref{eq:poisson}
$$
N_e = \sum_{n\in\NN^*} 2\pi\frac{m^*}{\hbar^2}\, \frac{e^{u}}{v}
e^{\epsilon_n v} |\chi_n|^2.
$$
We are then able to solve the Poisson equation \fref{eq:poisson}
with boundary condition \fref{eqbord:V}. Indeed,
due to the strong coupling of the entire system, the simple
resolution of equation \fref{eq:poisson} does not provide a
converging algorithm. Following \cite{caussignac} a Gummel iteration
algorithm is used, amounting to compute the new potential $V_{new}$
by solving the following modified Poisson equation
\begin{equation}
\mbox{div}(\eps_R\nabla V_{new}) + \frac {e}{\eps_0} N_e \frac
{V_{new}} {V_{ref}} = \frac {e}{\eps_0} \left(N_e(1-
\frac{V_{old}}{V_{ref}})- N_D\right), \quad \mbox{in } \Omega,
\end{equation}
with $V_{ref}=k_B T_L/e$.
\item We repeat the last three steps until the quantity
$\|V_{old}-V_{new}\|_{L^\infty}$
becomes sufficiently small.
Once the convergence is reached, we increment the applied 
drain-source bias $V_{DS}$ of $0.02$ V and start a new iteration.
\end{enumerate}

\subsection{Numerical results}

In this section we present and comment the performed numerical
results. The modeled device is schematized in Figure \ref{dispo} and
the physical values are chosen as the ones in
\cite{ddspnum} which are recalled in Table \ref{table}. We take
$N_x=50$ points in the transport direction and $N_z=50$ in the
confined direction for all the tests. The results are presented for
$\bet=\frac 12$ in \fref{h1}, which corresponds to
the so-called Chen model \cite{chenetal}.
\begin{table}
\caption{Table of the main physical values}
\label{table}
\medskip
\begin{center}
\begin{tabular}{|c||c|||c||c|}
\hline
\quad Parameter \quad  & Value & \quad Length\quad \ &\quad Value \quad \  \\
\hline
\hline
$N^+$ & $\quad 10^{26} m^{-3}\quad $ & $L_S$ & $10 nm$ \\
\hline
$N^-$ & $10^{21} m^{-3}$ & $L_C$ & $30 nm$ \\
\hline
$U_c$ & $3\ eV$ & $L_D$ & $10 nm$ \\
\hline
$\varepsilon_R[Si]$ & $11.7$ & $\ell_{ox}$ & $3 nm$ \\
\hline
$\varepsilon_R[SiO_2]$ & $3.9$ & $\ell_{Si}$ & $5 nm$ \\
\hline
\end{tabular}
\end{center}
\end{table}

Some other physical coefficients should be determined. The effective mass
is $m^*=0.19\, m_e$ with $m_e$ the electron mass, the lattice temperature is
$T_L=300\, K$ and the scattering coefficient for the elastic
collisions  is  the one used in \cite{ddspnum}
$\phi^0=\frac{1}{\mu^0 n_i}$, where the low  field mobility is taken
as $\mu^0=0.12\, m^2 V^{-1} s^{-1}$ and the intrinsic density is given
by $n_i=10^{10}\, m^{-2}$. We have to fix the value of the scattering
coefficient $\phi^{ph}$ for the electron-phonon interaction.
% We expect that the $I-V_{DS}$
% characteristics reaches a saturation value when $V_{DS}$ becomes large
% (at least for not too big Gate voltage $V_G$).
As noticed in subsection \ref{ddsection}, when $\phi^{ph} \to +\infty$, the model converges formally
to the subband drift-diffusion system presented in \cite{ddsp,ddspnum}.
It is then interesting to compare the numerical results for large and small values
of $\phi^{ph}$.
Figure \ref{DDET} displays the $I-V_{DS}$ characteristics
for $\phi^{ph}=10^{-4}/\phi^0$ and $\phi^{ph}=10^5/\phi^0$.
As expected and as noticed in \cite{baccarani}, the energy-transport
model gives higher currents compared with the drift-diffusion model,
due to the electron velocity overshoot within the channel.
Figure \ref{DDET} (right) shows the temperature for $\phi^{ph}=10^5/\phi^0$,
confirming that we are in the drift-diffusion regime, where the temperature
is constant. Indeed we check numerically  
that $\max T - \min T=2.6584\, 10^{-7}$ K.

\begin{figure}[!ht]
\includegraphics*[width=7.7cm]{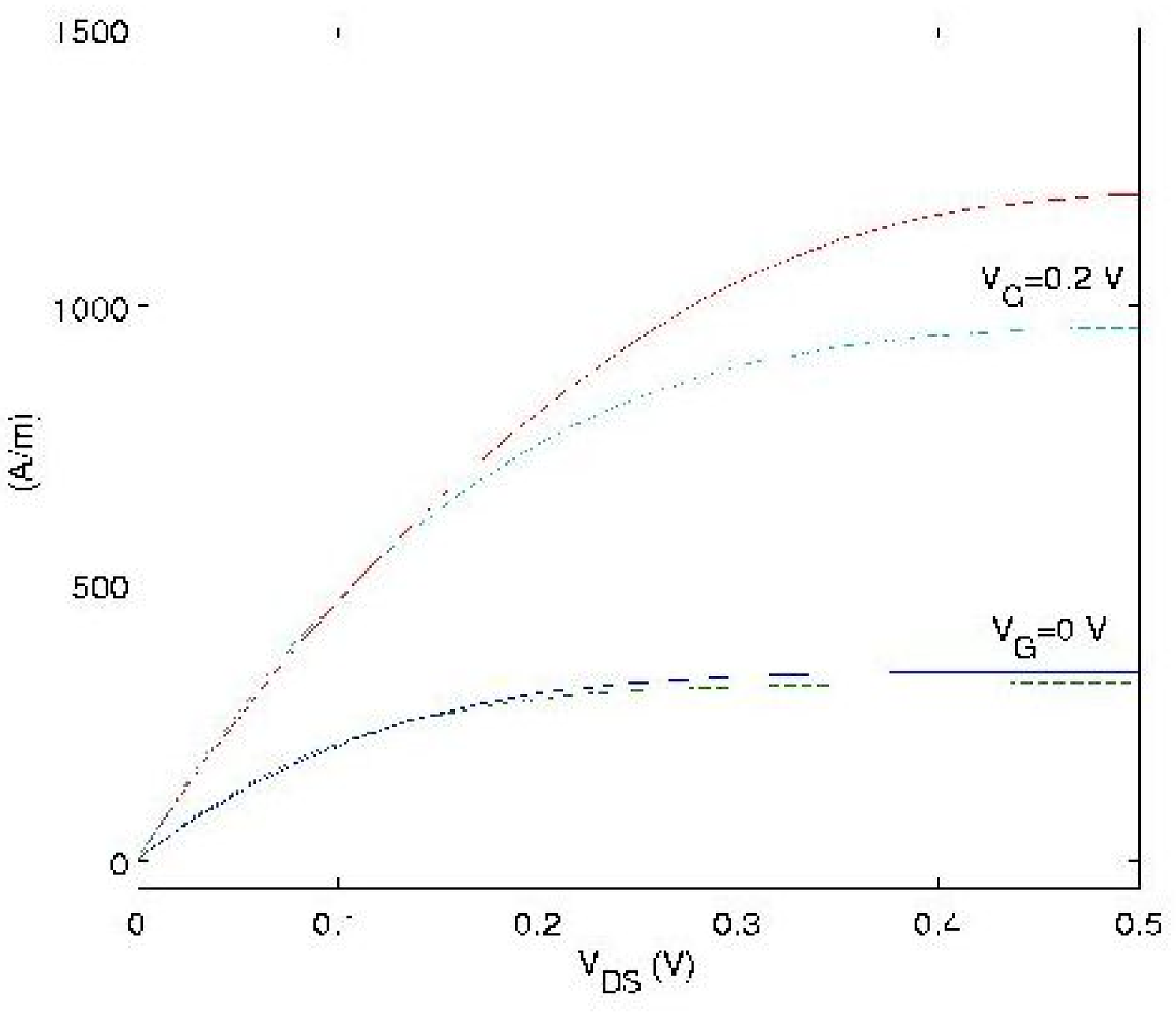}
\includegraphics*[width=7.7cm]{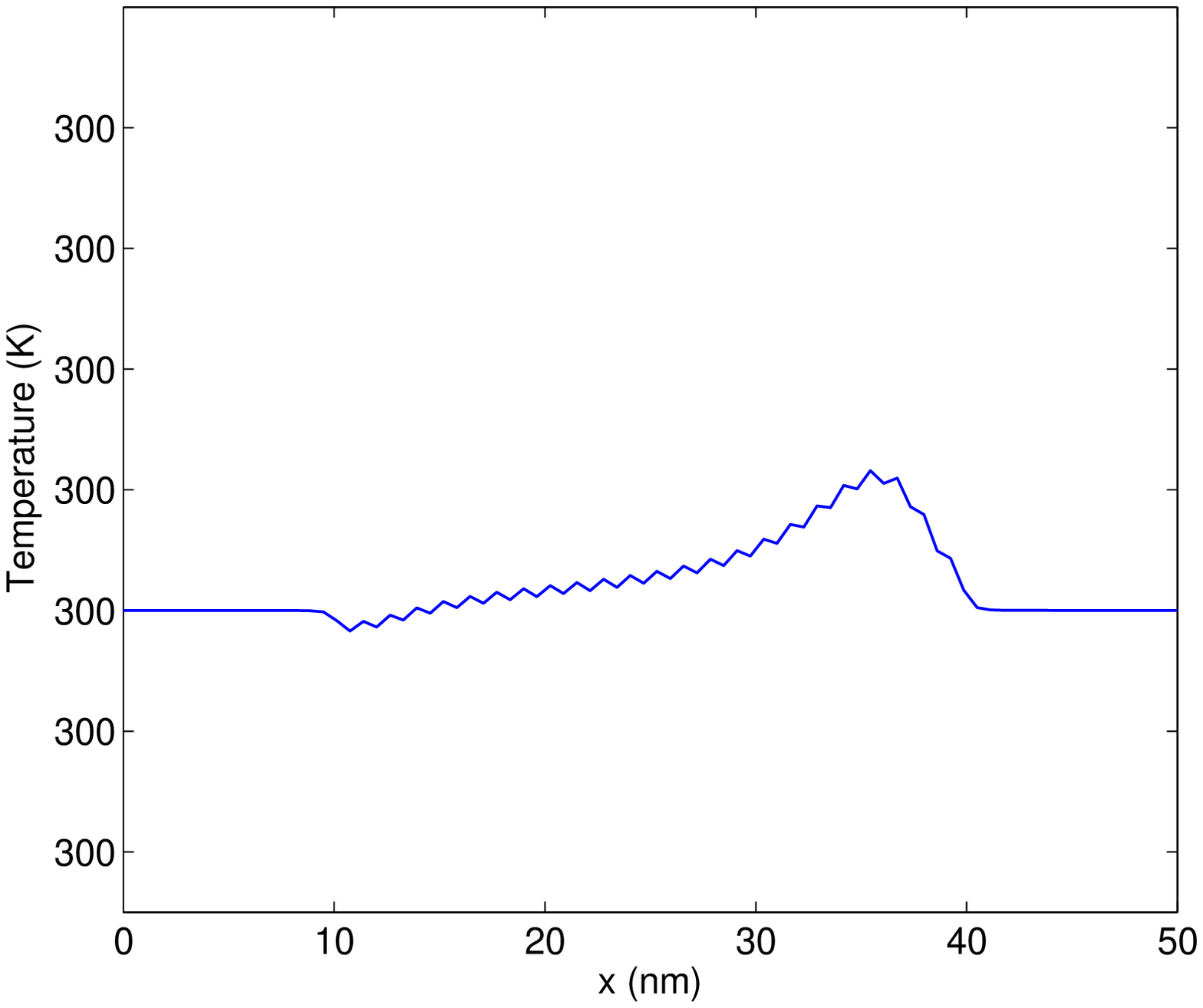}
\caption{{\it Left} : $I-V_{DS}$ characteristics for $V_G= 0 V$ and $V_G= 0.2 V$.
The dashed line corresponds to $\phi^{ph}=10^{-4}/\phi^0$, the solid
line corresponds to $\phi^{ph}= 10^5/\phi^0$ which is a good
approximation to the drift-diffusion model.
{\it Right} : Temperature in the device for $\phi^{ph}= 10^5/\phi^0$. We see that
the temperature is almost constant.}
\label{DDET}
\end{figure}

In the rest of the section, we present the results only for 
$\phi^{ph}=10^{-4}/\phi^0$ which corresponds to the energy-transport regime.
Figure \ref{IV} displays the computed current vs drain-source applied bias
characteristics with this chosen value.
We present in Figure \ref{IV} (left) the characteristics for different
numerical values of the Gate voltage $V_G$ and with $\ell_{Si}=5$ nm.
In Figure \ref{IV} (right), we display the characteristics for $V_G = 0$ V
and for different geometry of the devices : $\ell_{Si}= 4$, $5$ or $7$ nm.
These characteristics are comparable to the one obtained in
\cite{baccarani, claudia, ddspnum}.
We present in Figure \ref{temp} the evolution of the temperature in
the device with respect to the drain-source voltage for two different
values of the Gate voltage.
Figure \ref{meanvelo} displays the evolution of the mean velocity,
defined by $J_1/(q \rho_{\mu,T})$, where the one dimensional density is 
given in \eqref{rhoexpl} and the current in \eqref{J1DD}.
As expected, we notice an overshoot of the velocity at the
frontier between the channel and the drain for high value of $V_{DS}$.

\begin{figure}[!ht]
\includegraphics*[width=7.7cm]{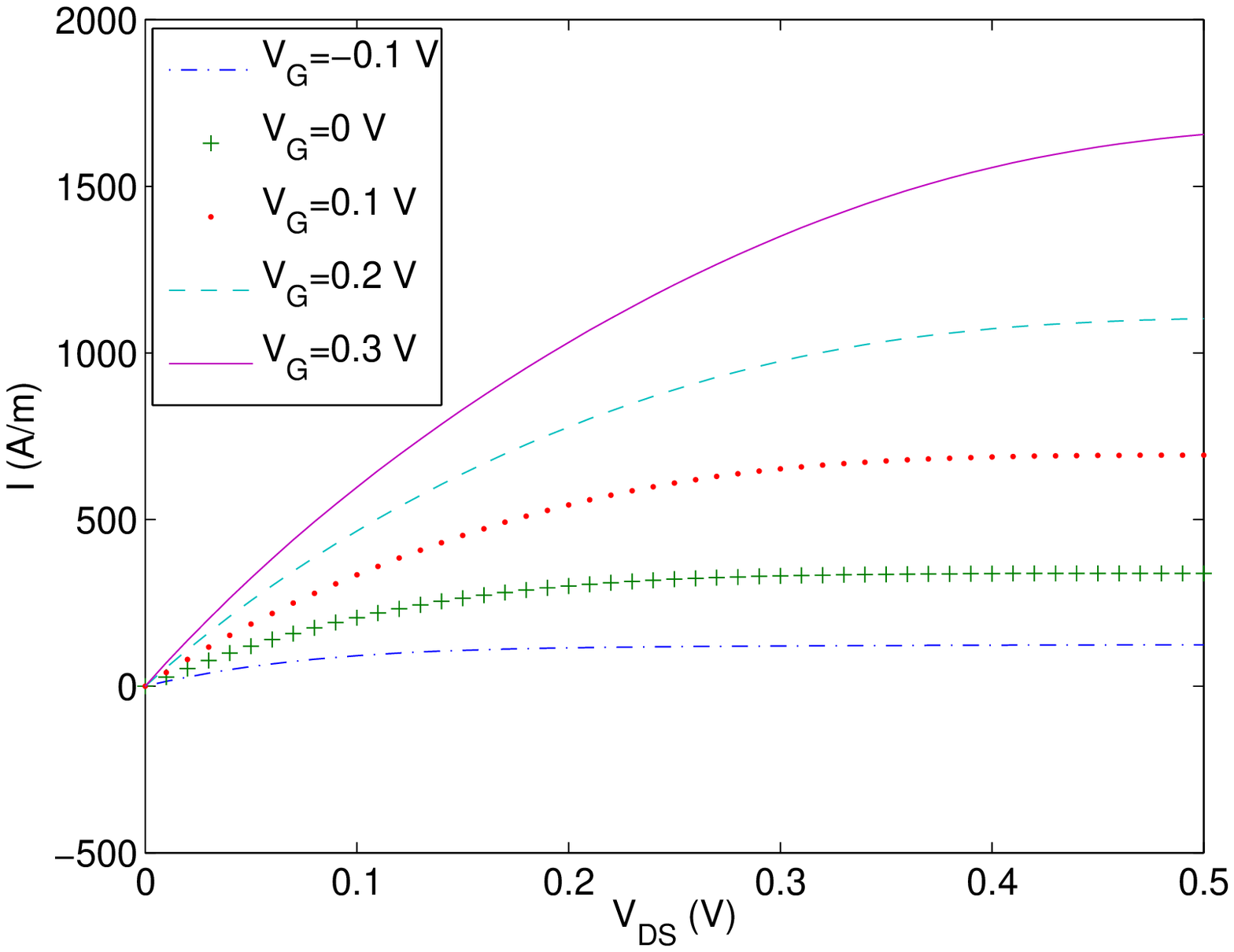}
\includegraphics*[width=7.7cm]{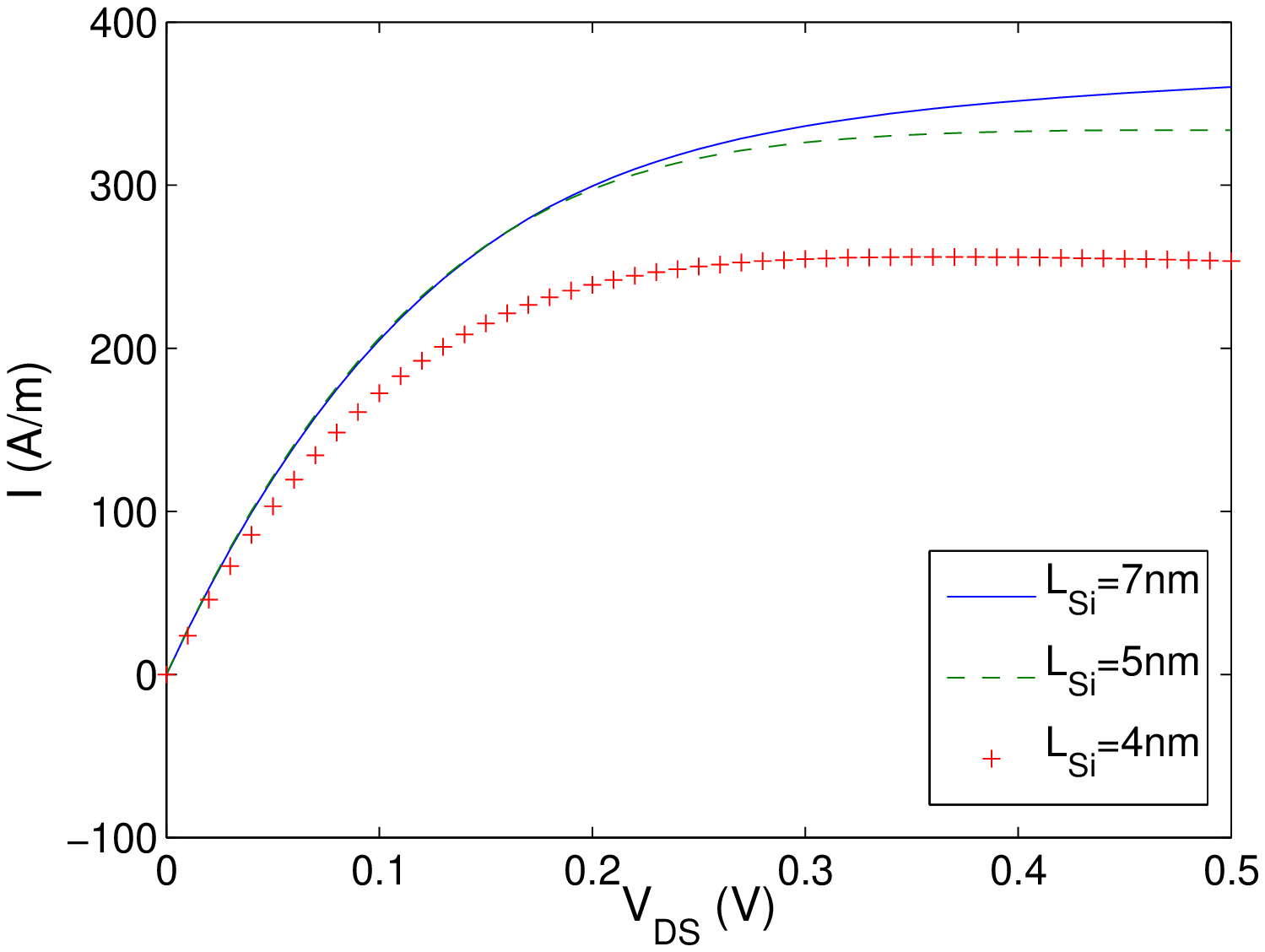}
\caption{$I-V_{DS}$ characteristics for different Gate voltages $V_G$ with 
$\ell_{Si}=5 nm$ (left) and for different width of the Silicon in the 
DG-MOSFET with $V_{G}=0$ V (right).}
\label{IV}
\end{figure}

\begin{figure}[!ht]
\includegraphics*[width=7.7cm]{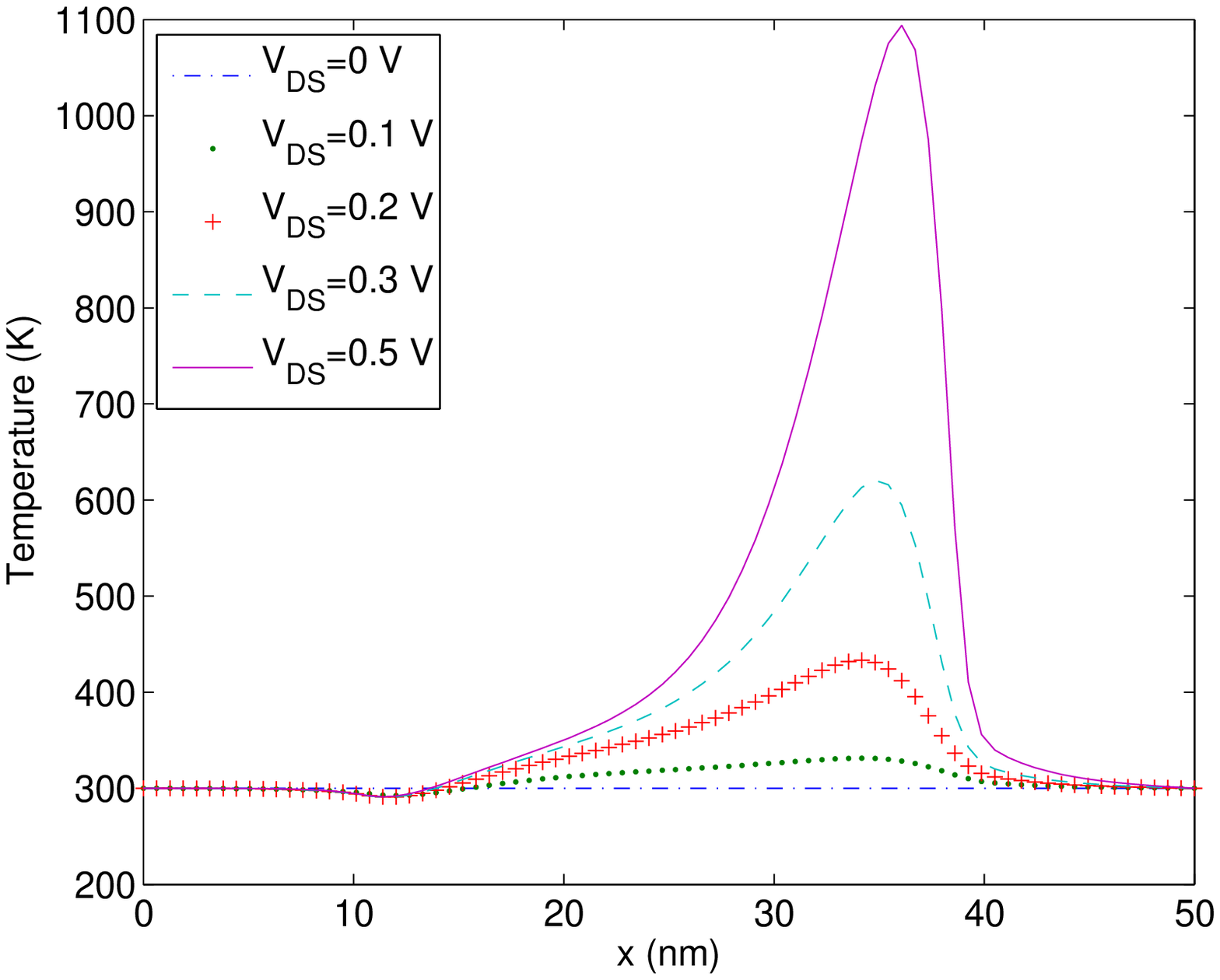}
\includegraphics*[width=7.7cm]{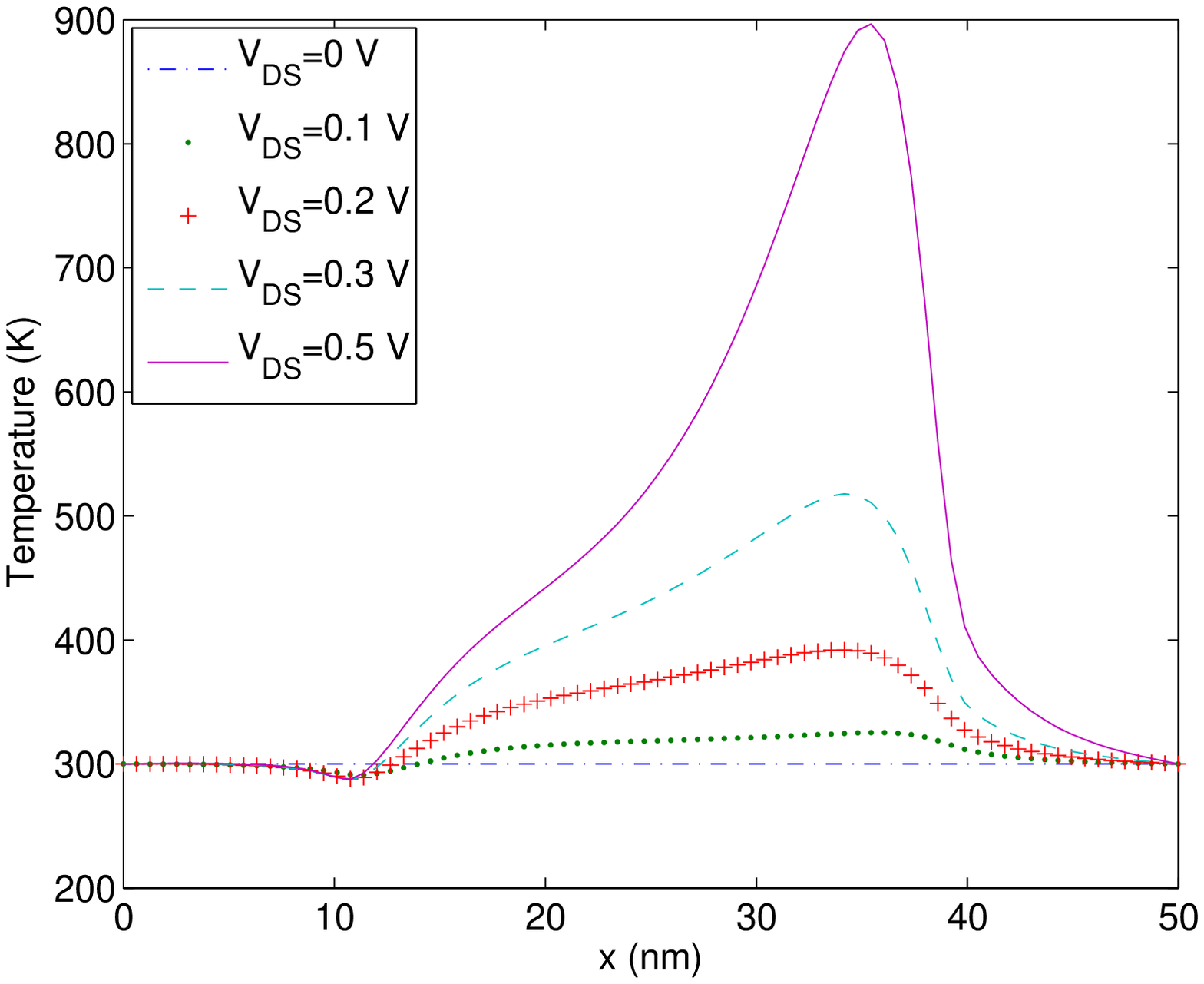}
\caption{Evolution of the temperature in the device for a Gate voltage
$V_G=0 V$ (left) and $V_G=0.2 V$ (right).}
\label{temp}
\end{figure}

\begin{figure}[!ht]
\includegraphics*[width=7.5cm]{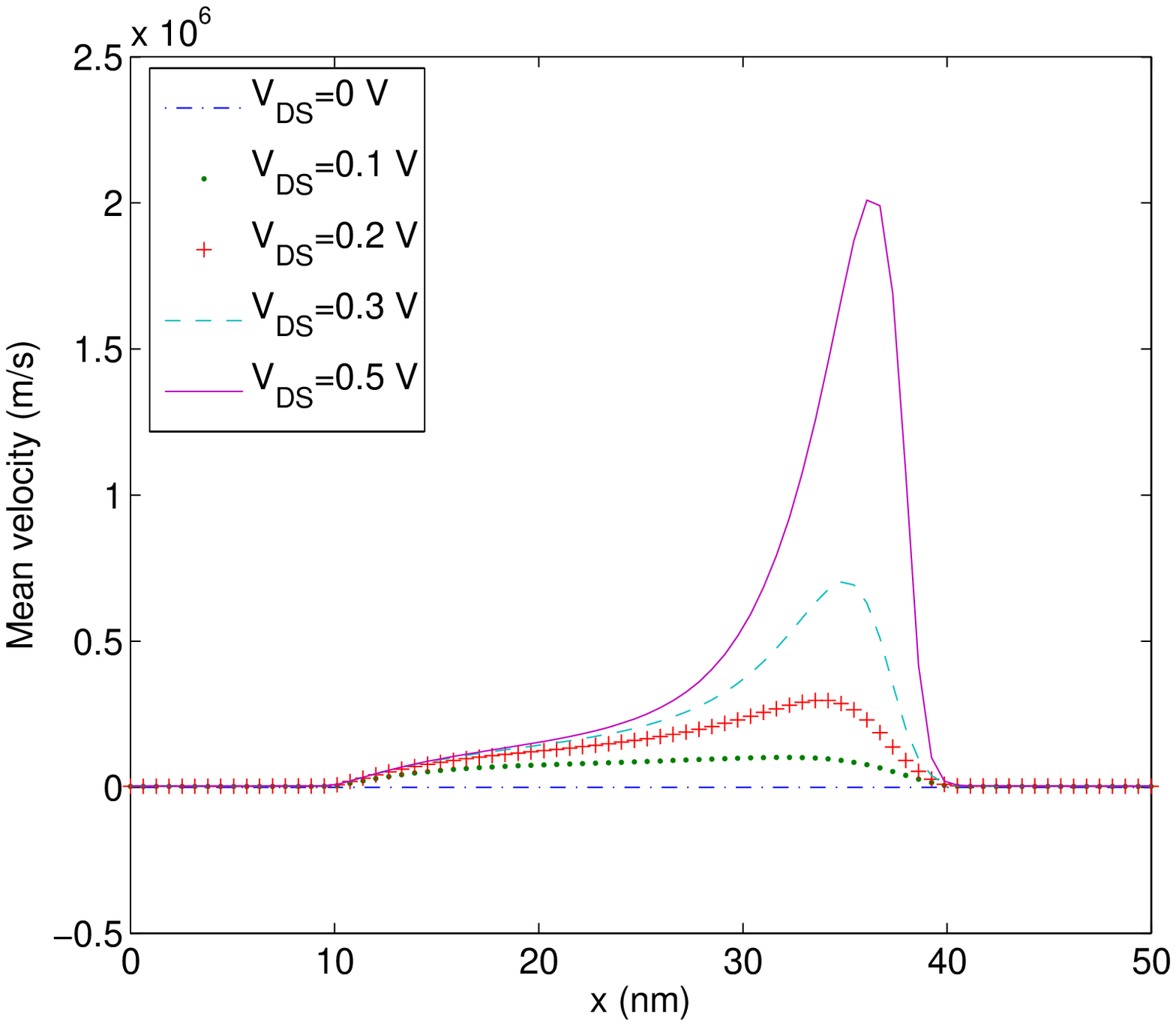}
\includegraphics*[width=7.5cm]{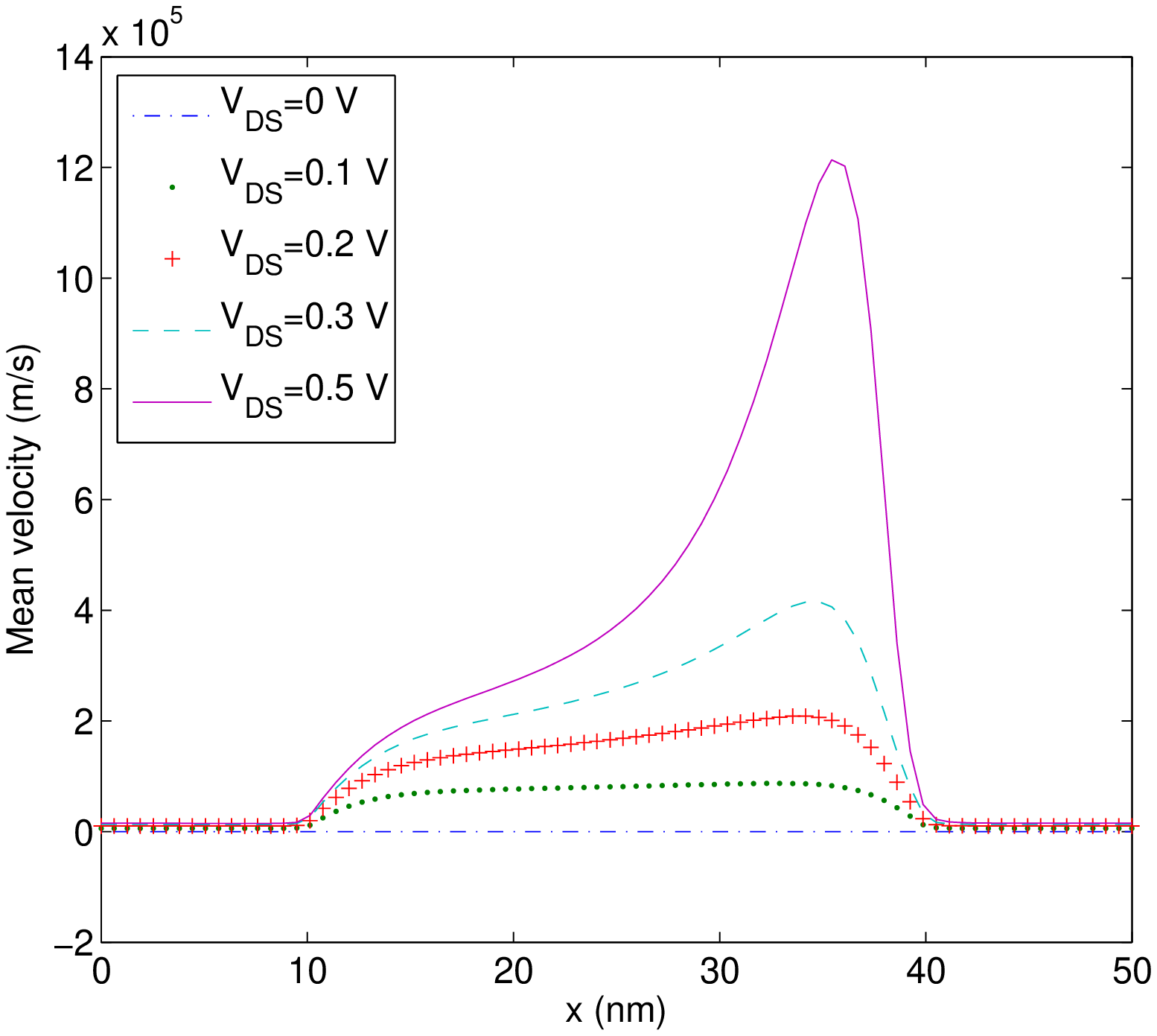}
\caption{Mean velocity for different drain-source potentials $V_{DS}$
and for $V_G=0 V$ (left) and $V_G=0.2 V$ (right).}
\label{meanvelo}
\end{figure}

We plot in Figures \ref{density} the shape of the density in the
device for two differents drain-source voltage. For $V_{DS}=0 V$,
we are at equilibrium and the density is symmetric in the device.
For $V_{DS}=0.5 V$, we notice transport of the charge carriers
in the device.

\begin{figure}[!ht]
\includegraphics*[width=7.7cm]{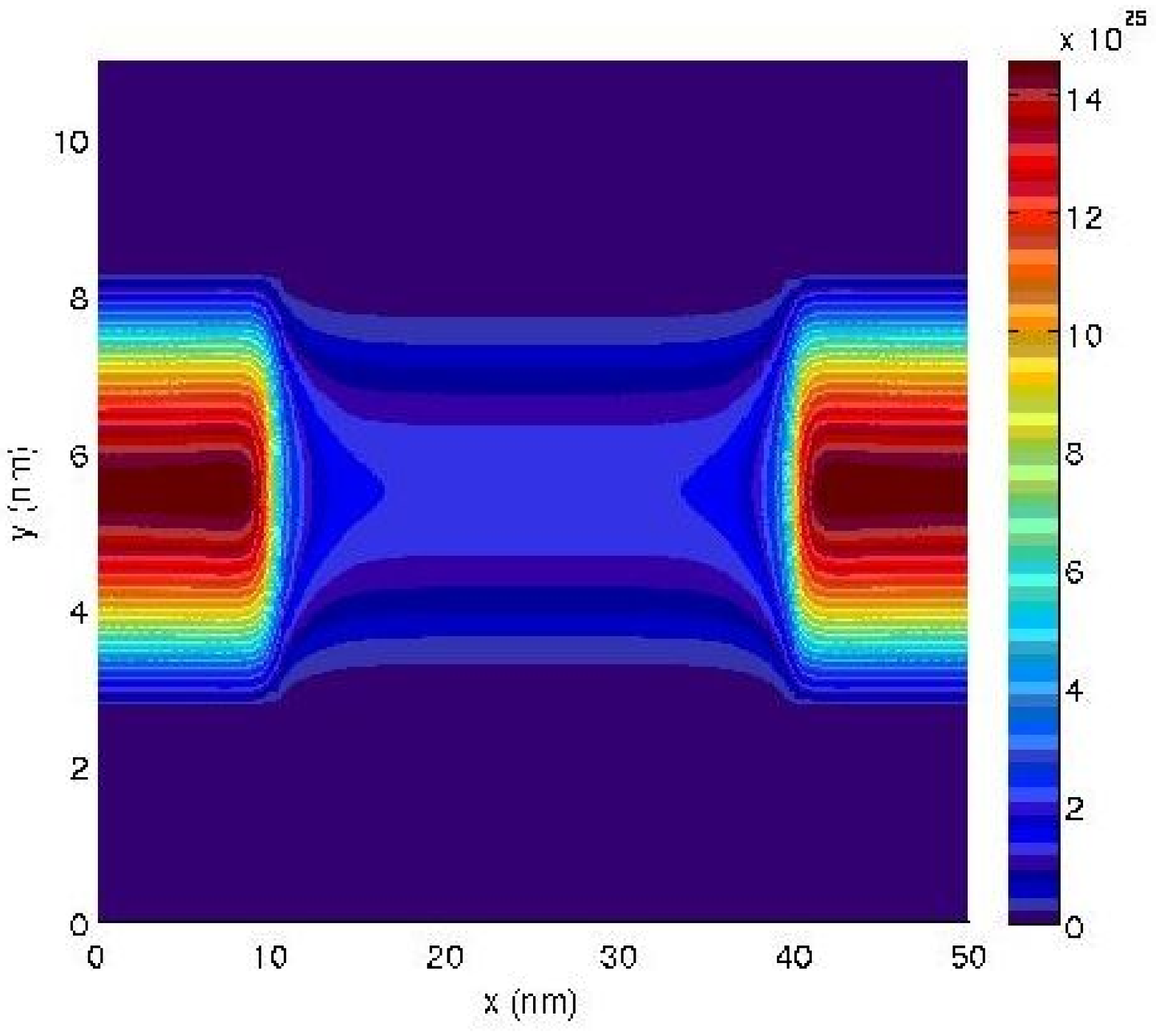}
\includegraphics*[width=7.7cm]{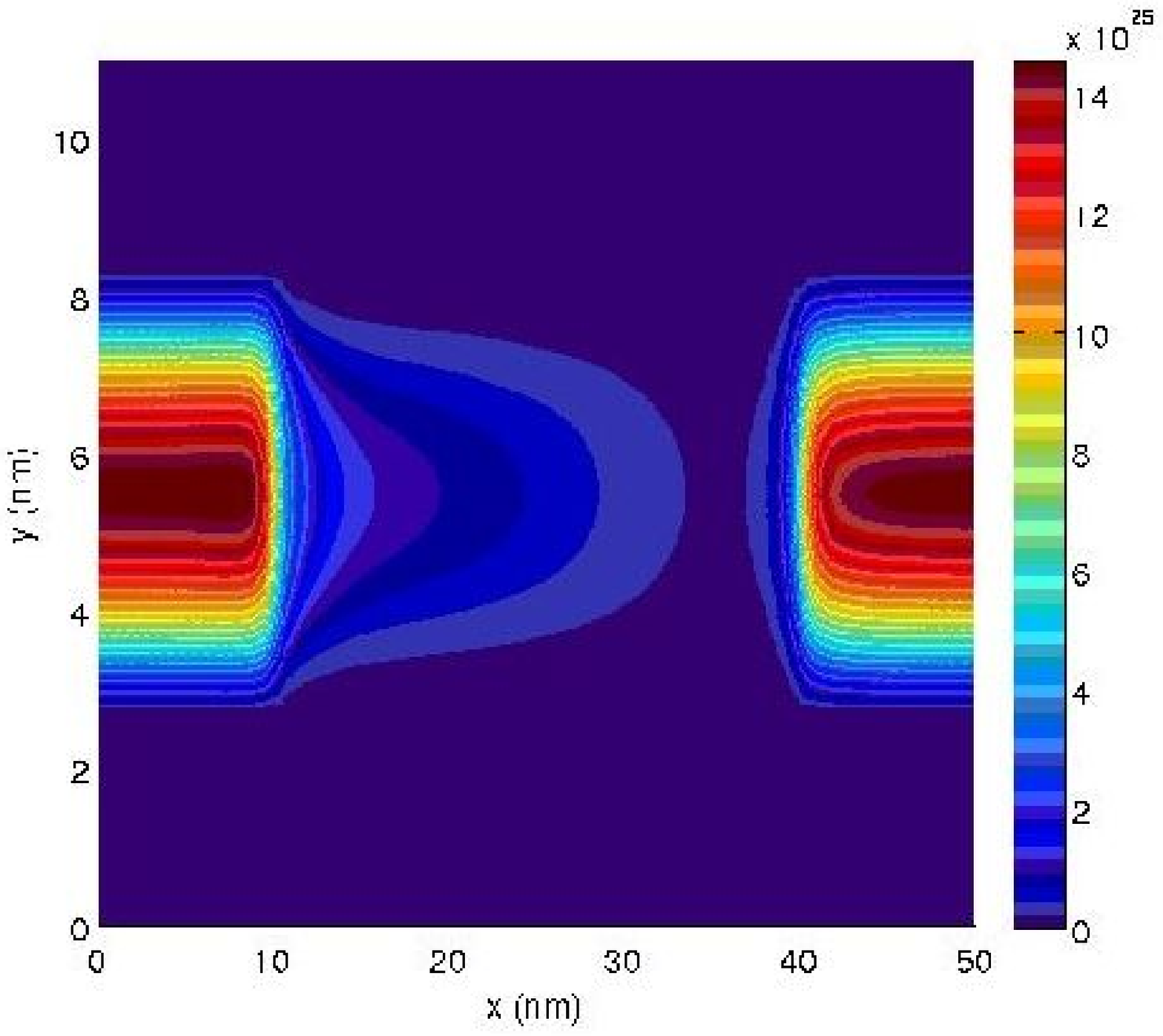}
\caption{Density of electrons in the device for $V_{DS}= 0 V$ (left)
and $V_{DS}=0.5 V$ (right); in this simulation we take $V_G=0.2 V$.}
\label{density}
\end{figure}

\section{Conclusion}

A coupled quantum--classical model has been obtained for describing
the transport of a partially confined electron gas. In a subband
decomposition framework, the transport model is obtained by means of
diffusive approximation from adiabatic quantum-kinetic models. The
final system in the transport direction is obtained through two
steps. First, under the assumption of dominant elastic scattering, a
SHE system is derived (referring to \cite{nbafmnegulescu}).
Then, under dominant electron--electron collisions, an energy transport
model is given, obtaining diffusion coefficients well suited for numerical purposes
and with a relaxation term taking into account the electron-phonon interactions.
In particular, with
energy dependent cross--section of the collision operator explicit 
expression of the diffusion
coefficients and of the relaxation term is derived and used for the
numerical simulation of transport in a Double-Gate MOSFET.
In the limit of large electrons-phonons collisions, we recover
the model of \cite{ddspnum}.
We point out that a saturation of the current is observed without need
of resorting at mobility modeling as done in \cite{ddspnum}.

\bigskip
\noindent {\bf Acknowledgements.} The authors acknowledge partial
support from the Galil\'ee Project no 25992ND of the Hubert Curien program~:
{\it ``Mod\`eles num\'eriques du
transport collisionnel dans des dispositifs nano--\'electroniques''}.

%%%%%%%%%%%%%%%%%%%%%%%%%%%%%%%%%%%%%%%%%%%%%%%%%%%%%%%%%%%%
%%%  Bibliography  %%%
%%%%%%%%%%%%%%%%%%%%%%%%%%%%%%%%%%%%%%%%%%%%%%%%%%%%%%%%%%%%

%
\end{document}